\documentclass{amsart}

\reversemarginpar

\usepackage{amsmath, amsthm, amssymb}

\newcommand{\lb}{\langle}
\newcommand{\rb}{\rangle}

\newcommand{\rrestriction}{\!\!\restriction}

\newcommand{\im}[1]{\mbox{Im}(#1)}

\newcommand{\desc}{\mbox{${\rm desc}$}}
\newcommand{\anc}{\mbox{${\rm anc}$}}
\newcommand{\aut}{\mbox{${\rm Aut}$}}

\usepackage{graphicx}

\usepackage{enumerate}
\usepackage{verbatim}

\newcommand{\Aut}{\operatorname{Aut}}

\newcounter{case}
\newenvironment{case}
{\stepcounter{case} \noindent \textbf{Case~\arabic{case}:}\itshape}{}

\newcounter{subcase}[case]

\newcounter{subsubcase}[subcase]

 \theoremstyle{plain}
 \newtheorem{thm}{Theorem}[section]
 \numberwithin{equation}{section} 
 \numberwithin{figure}{section} 
 \theoremstyle{plain}
 \newtheorem{lem}[thm]{Lemma} 
 \newtheorem{prop}[thm]{Proposition}

 \newtheorem{corol}[thm]{Corollary} 
 \theoremstyle{definition}
 \newtheorem{defn}[thm]{Definition}
 \theoremstyle{remark}

\usepackage{epic,eepic,ecltree} 

\input xy
\xyoption{all}

\usepackage{xy} 
\newcommand{\dedge}[1]{\ar@{--}[#1]}
\newcommand{\edge}[1]{\ar@{-}[#1]}
\newcommand{\arcc}[1]{\ar@{->}[#1]}
\newcommand{\arccbend}[1]{\ar@/^2pc/[#1]}
\newcommand{\arccbendunder}[1]{\ar@/_2pc/[#1]}
\newcommand{\arcb}[1]{\ar@{<-}[#1]}
\newcommand{\lulab}[1]{\ar@{}[l]_<<{#1}}
\newcommand{\rulab}[1]{\ar@{}[r]^<<{#1}}
\newcommand{\ldlab}[1]{\ar@{}[l]^<<{#1}}
\newcommand{\rdlab}[1]{\ar@{}[r]_<<{#1}}
\newcommand{\node}{*+[o][F-]{ }}

\begin{document}

\title[]{Locally-finite connected-homogeneous digraphs   \\ \today}


\maketitle

\vspace{-4mm}

\begin{center}

    ROBERT GRAY\footnote{This work was supported by an EPSRC Postdoctoral Fellowship EP/E043194/1 held by the first author at the School of Mathematics \& Statistics of the University of St Andrews, Scotland. \\
    \indent This author was partially supported by FCT and FEDER, project POCTI-ISFL-1-143 of Centro de \'{A}lgebra da Universidade de Lisboa, and by the project PTDC/MAT/69514/2006.}

    \medskip

    Centro de \'{A}lgebra da Universidade de Lisboa, \\ Av. Prof. Gama Pinto, 2,  1649-003 Lisboa, Portugal.

    \medskip

    \texttt{rdgray@fc.ul.pt}

    \bigskip
    
    \medskip

    R{\"O}GNVALDUR G. M{\"O}LLER

    \medskip

    Science Institute,\ University of Iceland, \\ 
    Dunhaga 3, 107 Reykjavik, Iceland.

    \medskip

    \texttt{roggi@raunvis.hi.is} \\
\end{center}
\begin{abstract} 
A digraph is connected-homogeneous if any isomorphism between finite connected induced subdigraphs extends to an automorphism of the digraph. We consider locally-finite connected-homogeneous digraphs with more than one end. In the case that the digraph embeds a triangle we give a complete classification, obtaining a family of tree-like graphs constructed by gluing together directed triangles. In the triangle-free case we show that these digraphs are highly arc-transitive. We give a classification in the two-ended case, showing that all  
examples arise from a simple construction given by gluing along a directed line copies of some fixed finite directed complete bipartite graph. 
When the digraph has infinitely many ends we show that the descendants of a vertex form a tree, and the reachability graph (which is one of the basic building blocks of the digraph)
is one of: an even cycle, a complete bipartite graph, the complement of a perfect matching, or an infinite semiregular tree. We give examples showing that each of these possibilities is realised as the reachability graph of some connected-homogeneous digraph, and in the process we obtain a new family of highly arc-transitive digraphs without property $Z$. 
\end{abstract}

\maketitle



\section{Introduction}
\label{sec_intro}

By a digraph we mean a set with an irreflexive antisymmetric binary relation defined on it. So a digraph $D$ consists of a set  of vertices $VD$ together with a set of pairs of vertices $ED$, called \emph{arcs}, such that there are no loops (arcs between a vertex and itself) and between any pair of vertices we do not allow arcs in both directions. A digraph is called \emph{homogeneous} if any isomorphism between finite induced subdigraphs extends to an automorphism of the digraph. The finite homogeneous digraphs were classified by Lachlan in \cite{Lachlan3} and, in a major piece of work, Cherlin in \cite{Cherlin1} classified the countably infinite homogeneous digraphs.

Various other symmetry conditions have been considered for digraphs. In \cite{Praeger1} and \cite{Cameron2} the class of \emph{highly arc-transitive} digraphs was investigated. For any natural number $k$ a $k$-arc in a digraph $D$ is a sequence $(x_0, \ldots, x_k)$ of $k+1$ vertices of $D$ such that for each $i$ ($0 \leq i < k$) the pair $(x_i,x_{i+1}) \in ED$. A digraph $D$ is said to be $k$-arc-transitive if given any two $k$-arcs $(x_0,\ldots,x_k)$ and $(y_0,\ldots,y_k)$ there is an automorphism $\alpha$ such that $\alpha(x_i) = y_i$ for $0 \leq i \leq k$, and $D$ is said to be \emph{highly arc-transitive} if it is $k$-arc-transitive for all $k \in \mathbb{N}$. 
In particular it follows that in a $k$-arc-transitive digraph the subdigraphs induced by any pair of $k$-arcs are isomorphic to each other. Specifically, if $D$ is a connected infinite locally-finite digraph then $D$ will have a $k$-arc $(x_0, \ldots, x_k)$ whose induced subdigraph only contains the arcs $(x_i, x_{i+1})$, and thus if in addition $D$ is assumed to be $k$-arc-transitive then all $k$-arcs in $D$ will have this form.  
For undirected graphs high-arc-transitivity is not an interesting notion since the only highly arc-transitive undirected graphs are cycles or trees, but for digraphs the family is very rich, and is still far from being understood. Following \cite{Cameron2} several other papers have been written on this subject; see \cite{Evans1}, \cite{Malnic2}, \cite{Malnic1}, and \cite{Moller6} for example. In contrast to the results on homogeneity described in the previous paragraph, the property of being highly arc-transitive is not restrictive enough for an explicit description of all possible countably infinite examples to be obtained.

In this paper we consider a variant of homogeneity where we only require that isomorphisms between \emph{connected} substructures extend to automorphisms. 
A digraph is \emph{connected-homogeneous} if any isomorphism between finite induced \emph{connected} subdigraphs extends to an automorphism of the digraph. Of course, any homogeneous digraph is connected-homogeneous, but the converse is far from being true. For example any digraph tree with fixed in- and out-degree is connected-homogeneous, but is not homogeneous, since its automorphism group is not transitive on non-adjacent pairs. This notion was first considered for (undirected) graphs in \cite{Gardiner2} and \cite{Enomoto1}, where the finite and locally-finite connected-homogeneous graphs were classified. More recently 
this has been extended to arbitrary countable graphs in \cite{Gray1}. Part of the motivation for  \cite{Gray1} came from the fact that the undirected connected-homogeneous graphs are a subclass of the distance-transitive graphs; in the sense of \cite{BrouwerBook}. For digraphs we shall see that 
there is an analogous relationship between connected-homogeneity and high-arc-transitivity. Specifically, for a large family of digraphs, connected-homogeneity actually implies high-arc-transitivity, so this notion gives rise to a natural subclass of the highly arc-transitive digraphs.

Our ultimate aim is to classify the connected-homogeneous countable digraphs.  In general this could be a difficult problem, especially in light of the fact that it would generalize Cherlin's result \cite{Cherlin1} for homogeneous digraphs. For undirected graphs by far the easiest part of the classification is in the infinite locally-finite case (where \emph{locally-finite} means that all vertices have finite degree). Indeed, the infinite locally-finite distance-transitive digraphs were classified in \cite{Macpherson1} and since connected-homogeneity implies distance-transitivity this deals with the locally-finite case, for undirected graphs. Motivated by this, in this paper we concentrate on the class of locally-finite connected-homogeneous digraphs.

When working with locally-finite graphs and digraphs there is a natural division into consideration of one-ended digraphs and of those with more than one end. Roughly speaking, the number of ends of a graph is the number of `ways of going to infinity', so a two-way infinite line has two ends, while an infinite binary tree has infinitely many (see Section~2 for a formal definition of the ends of a graph), and the number of ends of a digraph is the number of ends of its underlying undirected graph. A substantial theory exists for dealing with graphs with more than one end and therefore they are more 
tractable than one-ended graphs in most instances. 
For undirected graphs, connected-homogeneity actually implies that the graph must have more than one end in the locally-finite case (see \cite{Macpherson1}) and this is one of the reasons that this class is reasonably easy to handle. As we shall see below, a locally-finite  connected-homogeneous digraph need not have more than one end, so the one-ended case must be handled separately. Here we work exclusively on the case where the digraph has more than one end since, by analogy with undirected graphs, this case should be the most manageable. Even with this additional ends assumption, the family of digraphs obtained is very rich. In addition to the above motivation, this 
work may also be thought of as contributing to a general programme, initiated in \cite{Seifter1}, aimed at understanding the class of transitive digraphs with more than one end. It also provides yet another illustration of the usefulness of Dunwoody's theorem, and the associated theory of structure trees (in the sense of \cite{Dicks1}), for investigations of this kind. 

We now give a brief summary of our main results. After introducing the basic concepts in Section~2, we start our investigations by looking at many ended $2$-arc-transitive digraphs in Section~3. Under the assumption that the stabiliser of a vertex acts primitively on its in- and out-neighbours, if the digraph has more than one end we show its reachability graph must be bipartite, and if the digraph has strictly more than two ends then we prove that the descendants (and ancestors) of each vertex form a tree. These two results are proved using Dunwoody's theorem (stated in Section~2) and the associated theory of structure trees. In Section~4 we begin our study of connected-homogeneous digraphs, obtaining a structural result for the triangle-free case in Theorem~\ref{bigtheorem}. The majority of this section is devoted to proving part (iii) of this theorem, which is a classification of locally-finite connected-homogeneous bipartite graphs. Part (iv) of Theorem~\ref{bigtheorem} gives several families of infinitely ended connected-homogeneous digraphs. In Section~5 further constructions are described, showing that the examples of part Theorem~\ref{bigtheorem}(iv) on their own do not constitute a classification, in the triangle-free case. Applying results from Section~3, the $2$-ended case is dealt with in Section~6 where we show that the only examples are those given by taking the compositional product of a finite independent set with a two-way infinite directed line. Finally, in Section~7, we give a classification in the case that the digraph embeds a triangle, obtaining a family of digraphs built from directed triangles in a straightforward tree-like manner. Again, both in Sections~6 and 7, Dunwoody's theorem and the theory of structure trees are used extensively. 

\emph{Note added in proof.} Very recently in \cite{Hamann} it has been shown in by Hamann and Hundertmark that there are no other connected locally-finite connected-homogeneous digraphs with more than one end except those mentioned in this article. The main result of \cite{Hamann} is a classification of connected-homogeneous digraphs with more than one end, without any local finiteness assumption. Its proof makes use of a new theory of structure trees based on vertex cut systems, introduced recently by Dunwoody
and Kr\"{o}n \cite{DunwoodyKron}. 

\section{Preliminaries: ends, structure trees, reachability relations, and descendants}

Let $D = (VD, ED)$ be a digraph. 
Given vertices $x,y \in VD$ we write
$x \rightarrow y$ to mean $(x,y) \in ED$.  If every pair of distinct
vertices of $D$ are joined by an arc then $D$ is called a
\emph{tournament}. Given a digraph $D$ we say a vertex $w$ is an {\em out-neighbour} of a
vertex $v$ if $(v,w)\in ED$ and we define $D^+(v)=\{w\in D : (v,w)\in
ED\}$ and the out-degree of $v$ as
$d^+(v)=|D^+(v)|$.  Similarly we say a vertex $w$ is an in-neighbour of a
vertex $v$ if $(w,v)$ is in $ED$ and set $D^-(v)=\{w\in D : (w,v)\in ED\}$
and define the in-degree of a vertex $v$ as
$d^-(v)=|D^-(v)|$. We say that $D$ is
\emph{locally-finite} if every vertex has finite in- and
out-degree. The
\emph{neighbourhood} of $v$ is defined as the set 
$D(v) = D^+(v)\cup D^-(v)$  and the \emph{degree} of a vertex $v$ is
defined as $d(v) = |D(v)|$.  We use $\Gamma(D)$ to denote the
undirected underlying graph of $D$: so $\Gamma(D)$ is the undirected
graph obtained by replacing every arc of $D$ by an edge. A
\emph{walk} in $D$ is a finite sequence $v_0, v_1, \ldots, v_n$ 
of vertices where
$v_{i-1}$ and $v_{i}$ are joined by an edge 
in the graph $\Gamma(D)$ for $i=1, 2,
\ldots, n$. A
\emph{path} is a walk without repeated vertices. By a
\emph{directed path} (or a \emph{directed walk}) we mean a path (or a walk) 
$v_0, v_1, \ldots, v_n$ such that 
 $v_{i-1} \rightarrow v_i$ for  $i=1, 2,
\ldots, n$.  Let $u \in VD$. For any
non-negative integer $r$ let $D^r(u)$ be the set of vertices that may
be reached from $u$ by a directed path of length $r$. For a negative
integer $r$ let $D^r(u)$ be the set of vertices $w$ such that $u \in
D^{-r}(w)$. 
If $X$ is a subset of the vertex set of $D$ we use $\langle X \rangle$
to denote the subdigraph of $D$ induced by $X$ (the subdigraph induced by $X$ consists of
the vertices of $X$ together with all arcs that have both end vertices in $X$). 
Similarly, 
for a set $Y$ of arcs
in $D$ we let $\langle Y \rangle$ denote the subgraph consisting of
all the vertices which occur as end vertices of arcs in $Y$ together
with the arcs in $Y$.

In an undirected graph $\Gamma$ we use $\sim$ to denote adjacency
between vertices in the graph.
Given a connected undirected graph $\Gamma$ and two vertices $u,v \in
V\Gamma$ we use $d_{\Gamma}(u,v)$ to denote the length of a shortest
path from $u$ to $v$, calling this the distance between $u$ and $v$ in
$\Gamma$. A graph $\Gamma$ is called \emph{bipartite} if $V\Gamma$ can be
partitioned into two disjoint non-empty sets $X$ and 
$Y$ such that each edge in
$\Gamma$ has one end vertex in $X$ and the other one in $Y$. 
The partition $X \cup Y$ is
then called a \emph{bipartition} of $\Gamma$. By a \emph{bipartite digraph}
we mean a digraph $D$ whose vertex set can be written as a disjoint
union $VD = X \cup Y$, where every arc of $D$ is directed from $X$ to
$Y$.  

 We use $K_n$ to denote the complete
graph with $k$ vertices, $K_{m,n}$ denotes the complete bipartite
graph with parts of sizes $m$ and $n$, $C_n$ denotes the cycle with
$n$ vertices (i.e. a graph with vertex set $\{ 0, \ldots, n-1\}$ and
$i \sim j$ if and only if $|i-j| \equiv 1 \pmod{n}$). By the
\emph{complement of a perfect matching} on $2n$ vertices we mean the
bipartite graph with bipartition $X \cup Y$ where $|X| = |Y| = n$ with
a bijection $\eta :X \rightarrow Y$  and $(x,y) \in X \times Y$ an
edge if and only if $y \neq \eta(x)$. For each $n \in \mathbb{N}$ we use $CP_n$ to denote the complement of perfect matching with $2n$ vertices. 
We use $D_3$ to denote the
directed $3$-cycle: so this is the digraph with vertex set $\{ a,b,c
\}$ and $a \rightarrow b \rightarrow c \rightarrow a$. 
More generally, for $n \geq 3$ we use $D_n$ to denote the directed $n$-cycle. 

A \emph{tree}
is a connected graph without cycles. 
Every tree is bipartite with a
unique bipartition.  
We call a tree \emph{regular}
if all of its vertices have the same degree
and \emph{semiregular} if  
any two vertices in the same part of the bipartition have the
same degree as one another.   We call a digraph $D$ a tree if
the corresponding undirected graph is a tree.  

We use $\Aut{D}$ to denote the automorphism group of the digraph
$D$. The digraph $D$ is said to be \emph{vertex-transitive} if its
automorphism group acts transitively on the set of vertices $VD$. 

\subsection*{\boldmath Ends, $D$-cuts and structure trees}

The theory of structure trees is a powerful tool to investigate graphs
with more than one end. 
In this article extensive use is made of this theory.  In
this subsection we provide a brief overview of the ideas and results
that will be needed. The ideas presented in this subsection are drawn
from 
\cite{Dicks1}, \cite{Moller5} and \cite{Thomassen1}, to which we refer
the reader for more details.    

First we outline the ideas for undirected graphs and then indicate how
they will be applied in this paper for digraphs.  For the rest of this
section let $\Gamma$ be an
infinite connected locally-finite graph. By a \emph{ray} in $\Gamma$
we mean an infinite sequence $\{ v_i \}_{i \in \mathbb{N}}$ of
distinct vertices such that $v_i \sim v_{i+1}$ for all $i$. The
\emph{ends} of the graph $\Gamma$ are equivalence classes of rays
where two rays $\rho$ and $\sigma$ are said to be equivalent if there
is a third ray $\tau$ such that $\tau$ intersects each of $\rho$ and
$\sigma$ infinitely often. It is a straightforward exercise to check
that this is an equivalence relation on the set of all rays. Of
course, any connected infinite locally-finite graph has at least one
end. In fact, it is known that a vertex-transitive graph has $0$, $1$, $2$
or $2^{\aleph_0}$ ends (this applies to non-locally finite graphs
as well, see \cite{Moller3}). 

For a subset $e$ of $V\Gamma$ we define the \emph{co-boundary} $\delta
e$ of $e$ to be the set of edges $a \in E \Gamma$ such that one vertex
of $a$ belongs to $e$ while the other belongs to $e^*$ (where $e^*$
denotes the complement of $e$ in $V \Gamma$), and we call the subset
$e$ a \emph{cut} of $\Gamma$ if $\delta e$ is finite. Clearly $e$ is a
cut if and only if $e^*$ is a cut, since $\delta e = \delta e^*$. For
any ray $\rho$ of $\Gamma$ and any cut  $e$ the ray $\rho$ can
intersect only one of $e$ or $e^*$ infinitely often, so we may speak  
of a ray $\rho$ belonging to a cut $e$. Also, if a ray $\rho$ belongs
to a cut $e$, and $\sigma$ is a ray belonging to the same end as
$\rho$ then $\sigma$ also belongs to the cut $e$, so we may sensibly
talk about the ends that belong to a given cut. 

Let $G \leq \Aut(\Gamma)$. 
Clearly if $e$ is a cut, then so is its image $g e$ under
the action of $g$, for any $g \in G$. 
 Given a cut $e_0$ we write $E = Ge_0
\cup Ge_0^*$ to denote the union of the orbits of $e_0$ and $e_0^*$ under
this action.  

\begin{thm}[Dunwoody~\cite{Dunwoody2}] \label{DunwoodysThm}
Let $\Gamma$ be an infinite locally-finite connected graph with more than one end and let $G \leq \Aut(\Gamma)$.
Then $\Gamma$ has a cut $e_0$ 
such that with $E = Ge_0 \cup Ge_0^*$ we have:
\begin{enumerate}
\item the subgraphs induced by $e_0$ and $e_0^*$ are both infinite and
  connected; 
\item and for all $e,f \in E$ 
there are only finitely many $g \in E$ such that $e \subset g
  \subset f$ and 
\item and for all $e,f \in E$ one of the following holds:
	\[e \subseteq f, \quad e \subseteq f^*, \quad e^* \subseteq f, 
\quad e^* \subseteq f^*.	\]
\end{enumerate}
\end{thm}

We call a cut satisfying the conditions of Dunwoody's theorem a
\emph{$D$-cut} and the set $E$ is called a {\em tree set}.

We now show how to construct a graph theoretic tree $T=T(E)$ from
$E$ that we call a \emph{structure tree} for $\Gamma$. The graph
$T(E)$ will have directed edges that come in pairs 
$\{ (u,v), (v,u) \}$, and there will be a bijection between the tree
set $E$ and the set of arcs of $T(E)$. We think of this bijection as a
labelling of the arcs of $T(E)$ by the elements of $E$.  
The arcs in $T(E)$ will be labelled in such a way that if $e$ labels
the arcs $(u,v)$ then its complement  $e^*$ labels the reverse arc
$(v,u)$. We shall write $e= (u,v)$ to mean that $e \in E$ labels the
arc $(u,v)$. Note that $T$ is not a directed graph in the sense defined in Section~\ref{sec_intro} above, 
since the arc relation on $T$ is symmetric.     

To construct $T(E)$ we begin with the disjoint union $Y$ of oppositely
oriented pairs of arcs $\{ e,e^* \}$ labelled by the elements $e \in
E$ and their complements. Then we want to glue together these edges to form a tree. Formally
this is achieved by defining an equivalence relation $\approx$ on the
set of vertices of $Y$.  
The gluing process is designed to result in a graph that will reflect
the structure of the poset $(E,\subseteq)$. Given $e,f \in E$ we write  
\[
f \ll e \ \mbox{if and only if} \ f \subset e \mbox{ and there is no cut }g
\mbox{ in } E \mbox{ such that } f\subset g\subset e.
\]
Now given two edges $e = (u,v)$ and $f = (x,y)$ in $Y$ we write $v
\approx x$ if $x=v$ or if $f \ll e$. Using the properties listed in
Theorem~\ref{DunwoodysThm} it may be shown that $\approx$ is an
equivalence relation on the set of vertices of $Y$; for details see
\cite[Page~12]{Moller5} and \cite[Theorem~2.1]{Dunwoody4}. 
Now define $T(E)$ to be the the graph $Y / \approx$ obtained by
identifying the $\approx$-related vertices of $Y$. Since edges have not
been identified, the set of arcs of $T(E)$ is still in bijective
correspondence with the set $E$. The structure of $(E,\subseteq)$ is
reflected in $T(E)$ since, from its construction, for arcs $e$ and $f$
in $T(E)$ there is a directed edge path from $e$ to $f$ if and only if
$f \subseteq e$ in $(E,\subseteq)$. Again using the conditions listed
in Theorem~\ref{DunwoodysThm}, it may then be shown that $T(E)$ is a
connected and has no simple cycles of length greater than $2$; in
other words $T(E)$ is a tree (see \cite[Page~13]{Moller5} for
details). 

Next we want to define a mapping $\phi: V\Gamma \rightarrow VT$ from
the vertex set of $\Gamma$ to the vertex set of $T = T(E)$, that we
call the \emph{structure mapping}. Given $v \in V\Gamma$, let $e =
(x,y) \in E$ be an arc such that $v \in e \subseteq V\Gamma$ and where
$e$ is minimal in $(E,\subseteq)$ with respect to containing $v$, and
define $\phi(v) = y$. That $\phi$ is  well-defined 
 is proved using the properties listed in
Theorem~\ref{DunwoodysThm}; see \cite{Moller5} for details. In general  
$\phi$ need not be either surjective or injective; see
\cite{Moller5}. Now the subgroup $G \leq \Aut(\Gamma)$ acts on $E$ and thus $G$ acts on $Y$. The way that we have 
identified vertices in $Y$ is clearly covariant with the action of $G$ and thus 
$G$ acts on $T$ as a group of automorphisms. 
This action commutes with the
mapping $\phi$, so  
for any $g \in G$ and $v \in V\Gamma$ we have $\phi(g v) = g
\phi(v)$.  

Just to illustrate how these properties can be used we show that if
$G \leq \Aut(\Gamma)$ acts transitively on $\Gamma$ then $T$ will not have any leaves
(vertices of degree 1).  Suppose $e=(u,v)$ is an arc in $T$.
  Since the graph is locally finite and $G$ acts
transitively we can find an element $g\in G$ such that $g(\delta
e)$ is a subset of the set of edges of $\Gamma$ that are contained entirely within $e \subseteq V\Gamma$.
Then either $g(e)\subseteq e$ or $g(e)^*\subseteq e$
and we see that $v$ cannot have degree 1.

The structure mapping gives us a way to relate 
structural information about the graph $\Gamma$ to information about
the tree $T$.  Note that if $G$ acts vertex transitively on 
$\Gamma$ it does not necessarily follow that $G$ acts transitively on $T$, 
since $\phi$ is not necessarily
surjective. There is another map $\Phi$ that maps each end of $\Gamma$
either to an end of $T$ or a vertex in $T$ (see \cite{Moller5}).
This map need not be injective, but the preimage of an end in $T$
consists only of a single end in $\Gamma$.  Hence, for instance we see
that if $\Gamma$ has precisely two ends then $T$ will also have  two
ends and will be a line.

Since in this paper we work with directed graphs we must
explain how the ideas outlined above may be applied in this
context. By the ends of a digraph $D$ we simply mean the ends of the
underlying undirected graph $\Gamma(D)$ of $D$. Now $V\Gamma(D) = VD$
and the theory described above then applies to $\Gamma(D)$. So by
saying $e \subseteq VD$ is a $D$-cut we mean that as a subset of
$V\Gamma(D)$ it is a $D$-cut of $\Gamma(D)$. Clearly $\Aut(D)$ is a 
subgroup of $\Aut(\Gamma(D))$ and so we can set $G = \Aut(D) \leq \Aut(\Gamma(D))$ and
apply the above theory, letting $e_0$ be a $D$-cut and considering the tree set $E = G e_0 \cup G e_0^*$ and corresponding structure tree $T = T(E)$. 
The definition of the mapping
$\phi: VD \rightarrow VT$ is then inherited naturally since $VD =
V\Gamma(D)$. 

\smallskip

\begin{quote}
\emph{Throughout, unless otherwise stated, $D$ will denote an infinite locally-finite connected digraph with more than one end, $G = \Aut(D)$, $e_0$ will denote a fixed $D$-cut of $D$, $E = G e_0 \cup G e_0^*$ the associated tree set, $T = T(E)$ the structure tree, and $\phi:VD \rightarrow VT$ the corresponding structure map.} 
\end{quote}

\subsection*{Reachability relations and descendants}

An \emph{alternating walk} in a digraph $D$ is a sequence of vertices
$(x_1,\ldots, x_n)$ such that either $(x_{2i-1},x_{2i})$
and $(x_{2_i+1},x_{2i})$ are arcs for all $i$, or
$(x_{2i},x_{2i-1})$ and $(x_{2i},x_{2i+1})$ are arcs for all $i$. We
say that 
$e'$ is {\em reachable} from  $e$ if there is an alternating walk
$(x_1,\ldots,x_n)$ such that the first arc traversed is $e$ and the
last one is 
$e'$. This is denoted by
$e \mathcal{A} e'$. Clearly the relation $\mathcal{A}$ is an
equivalence relation on  $ED$. 
The equivalence class containing the arc $e$ is denoted by
$\mathcal{A}(e)$. 
Let $\lb \mathcal{A}(e) \rb$ denote the subdigraph of $D$ induced by $\mathcal{A}(e)$. If $D$ is 
$1$-arc-transitive, then the digraphs $\lb \mathcal{A}(e) \rb$, for $e \in ED$, are all isomorphic to
a fixed digraph, which will be denoted by $\Delta(D)$.
The following basic result about reachability
graphs was proved in \cite{Cameron2}.  

\begin{prop}\cite[Proposition~1.1]{Cameron2}\label{ReachabilityBipartite}
Let $D$ be a connected $1$-arc-transitive digraph. Then $\Delta(D)$ is
$1$-arc-transitive and connected. Further, either 
\begin{enumerate}
\item $\mathcal{A}$ is the universal relation on $ED$ and $\Delta(D) = D$, or
\item $\Delta(D)$ is bipartite.
\end{enumerate}
\end{prop}

In this paper in most instances we shall be working with digraphs $D$ were $\Delta(D)$ is bipartite. Although $\Delta(D)$ is a directed bipartite graph, we shall often identify it with its undirected underlying bipartite graph (from which the original directed bipartite graph may be recovered by orienting all the edges from one part of the bipartition to the other). 

A question that still remains open from \cite{Cameron2} is whether
there exists a locally-finite highly arc-transitive digraph for which
the reachability relation $\mathcal{A}$ is universal; see
\cite{Malnic2} and \cite{Seifter1}. We shall see below that for the
class of highly arc-transitive digraphs considered in this paper
$\mathcal{A}$ cannot be universal, and hence by
Proposition~\ref{ReachabilityBipartite}, $\Delta(D)$ will be bipartite.   

For a vertex $u$ in $D$ the set of \emph{descendants} of
$u$ is the set of all vertices $v$ such that there is a directed 
path in $D$ from $u$ to $v$. This set is denoted by $\mathrm{desc}(u)$. 
For
$A \subseteq VD$ we define 
$\mathrm{desc}(A) = \bigcup_{v \in A}{\mathrm{desc}(v)}$.
The set of \emph{ancestors} $\mathrm{anc}(v)$ of a vertex $v$ is the
set of those vertices of $D$ for which $v$ is a descendant.
We shall also frequently be interested in the subdigraph induced by this set of vertices which we shall also denote by $\mathrm{desc}(u)$, similarly for $\mathrm{anc}(v)$.
 It was shown in \cite{Moller6} that 
if $D$ is a locally finite infinite connected highly arc-transitive
digraph then  
for any directed line $L$ in $D$ 
the subdigraph induced by $\mathrm{desc}(L)$ is highly arc-transitive
and has more than one end.
Here by a directed line $L$ we mean 
a two-way infinite sequence $\ldots, v_{-1}, v_0, v_1, \ldots$ such that
$(v_i,v_{i+1})$ is an arc for every $i \in \mathbb{Z}$. 

Note that no generality is lost by assuming that the digraphs we
consider are connected, since for each of the symmetry conditions
under consideration any disconnected countable example would be isomorphic to a
finite (or countable) number of disjoint copies of one of its
connected components.

\section{$2$-arc-transitive digraphs}

Before attacking the classification of  
connected-homogeneous digraphs, we shall first
present some preliminary results for $2$-arc-transitive locally-finite digraphs
with more than one end. These digraphs were considered in
\cite{Seifter1}. 

The following lemma builds on \cite[Lemma~2.4]{Seifter1}. 

\begin{sloppypar}
\begin{lem}\label{biglemma}
Let $D$ be a locally-finite connected 
$2$-arc-transitive digraph with more than one end and $G=\aut(D)$. 
We let $e_0$ denote
a D-cut and $E=Ge_0\cup Ge_0^*$ the associated tree set. In addition  
$T=T(E)$ is the structure tree and $\phi:VD\rightarrow VT$ the structure map.
Then we have the following.
\begin{enumerate}[(i)]
\item \label{Seifter} For all $e \in Ge_0 \cup G {e_0}^{*}$ there is
  no $2$-arc $(a,b,c)$ with $a,c \in e$ and $b \in e^*$. 
\item There exists a positive integer $N$ such that for any directed path
$(v_0,v_1,\ldots,v_k)$ of length $k$ in $D$ we have $d_T(\phi(v_0),\phi(v_k)) = kN$. \label{distances}
\item \label{pathsagain} If $(v,x_0,x_1,\ldots,x_r)$ and $(v,y_0,y_1,\ldots,y_s)$ are directed paths based at $v$ and $\phi(x_r) = \phi(y_s)$ then $r=s$.
\item  If $u,v \in D$ and there is a directed path in $D$ from $u$ to $v$ of length $n$ then every directed path from $u$ to $v$ has length $n$. In particular, there are no directed cycles in $D$. \label{nocycles}

\end{enumerate}
\end{lem} \end{sloppypar}
\begin{proof}
Part (i) is proved in \cite[Lemma~2.4]{Seifter1}.

We prove part (ii) by induction on the length of the path. Given an
arc $v_0\rightarrow v_1$ we set $N = d_T(\phi(v_0), \phi(v_1))$.  Since
$D$ is arc-transitive this number $N$ does not depend on the choice of   the arc
$v_0\rightarrow v_1$ and $N>0$.
 Now suppose the result holds for all $j \leq k$ and consider a
directed path $(v_0,v_1,\ldots,v_k,v_{k+1})$ of length $k+1$. Certainly we have
$d_T(\phi(v_k),\phi(v_{k+1}))=N$ and $d_T(\phi(v_0),\phi(v_{k}))=kN$ by
induction hypothesis. Hence $d_T(\phi(v_0),\phi(v_{k+1})) \leq (k+1)N$. If  
$d_T(\phi(v_0),\phi(v_{k+1})) < (k+1)N$ then 
$\phi(v_{k-1})$ and $\phi(v_{k+1})$
belong to the same component of $T \setminus \{
\phi(v_{k}) \}$. Let $e$ denote the arc in $T$ that starts 
in $\phi(v_k)$ and contains $v_{k+1}$. Then
$(v_{k-1}, v_{k}, v_{k+1})$ is a $2$-arc
with properties contradicting part (i). 

Parts (iii) and (iv) follow immediately from part (ii). The last sentence of (iv) follows since if there were a directed cycle with initial and terminal vertex $v$ then the path $(v)$ is a path of length $0$ from $v$ to itself, while following the directed cycle gives a longer path.
\end{proof}

\begin{lem} \label{just2ends}
Let $D$ be a locally-finite connected 
$2$-arc-transitive digraph with more than one end and $G=\aut(D)$. 
We let $e_0$ denote
a D-cut and $E=Ge_0\cup Ge_0^*$ the associated tree set. In addition  
$T=T(E)$ is the structure tree and $\phi:VD\rightarrow VT$ the structure map. Let $u$ be a vertex in $D$.
 Then $D$ has two ends if and only if for every $r\in \mathbb{Z}$ the
 map $\phi$ is constant on $ D^r(u)$.  
\end{lem}
\begin{proof}
Assume first that $D$ has two ends.  Then $T$ has just two ends and is
a line.  Let $x$ and 
$y$ be vertices in $D^r(u)$ for some integer $r$.  
Suppose that $r$ is non-negative (the case that $r$ is negative may be dealt with using a dual argument). 
By
Lemma~\ref{biglemma} we see that 
$d_T(\phi(u), \phi(x))=d_T(\phi(u), \phi(y))\neq 0$.   Thus 
there are only two possibilities for $\phi(x)$,
one on each side of $\phi(u)$.  Suppose $x$ and $y$ have distinct
images under $\phi$.    If $z\in D^{-r}(u)$ then there is a
directed path of length $r$ from $z$ to $u$ and thus 
$d_T(\phi(u),\phi(z))=d_T(\phi(u), \phi(x))=d_T(\phi(u), \phi(y))$ by
Lemma~\ref{biglemma}(ii).  Hence $\phi(z)$ must be either
equal to $\phi(x)$ or $\phi (y)$, but that would contradict
Lemma~\ref{biglemma}(ii) because there are directed
paths of length $2r$ from $z$ to both $x$ and $y$ and thus
$\phi(z)\neq \phi(x)$ and $\phi(z)\neq \phi(y)$.

Assume now that if $r$ is an integer then $\phi(x) = \phi(y)$ for every $x,y
 \in D^r(u)$.  Note that if $x\in  D^+(u)$ and $y\in D^-(u)$ then
 $\phi(x)\neq \phi (y)$.  From the assumptions and
 Lemma~\ref{biglemma}(ii) we conclude that the
 images under $\phi$ of $\desc(u)$ and $\anc(u)$ all lie on a line $L$.
 If $v$ is some vertex, either
 in $\desc(u)$ or $\anc(u)$ then all the images under $\phi$ of 
$\desc(v)$ or $\anc(v)$ lie on $L$.  An easy induction now shows that
the whole image of $\phi$ is contained in $L$.  The group $G$ has just
 one orbit on the edges of $T$ and thus at most two orbits on the
 vertices of $T$.  The image of $\phi$ is an orbit of $G$.   
From this and the fact that $T$ has no leaves we
 conclude that $T=L$ and that $D$ has just two ends.
\end{proof}

Before stating the next theorem we need the following observation.

\begin{lem}\label{decandanc}
Let $D$ be a vertex transitive digraph, and let $u \in VD$. Then
$\mathrm{desc}(u)$ is a tree if and only if $\mathrm{anc}(u)$ is a
tree.
\end{lem}
\begin{proof}
We show that $\mathrm{desc}(u)$ is not a tree if and only if 
$\mathrm{anc}(u)$ is not a tree.

Suppose $\mathrm{desc}(u)$ is not a tree.  Suppose first that  
there is a directed cycle in 
$\mathrm{desc}(u)$.  By vertex transitivity we may assume that the
vertex $u$ is in this cycle.  Then this cycle is also contained in 
$\mathrm{anc}(u)$ which is then not a tree.  
Suppose now that
$\mathrm{desc}(u)$ has a cycle that is not a directed cycle.  
Hence there are distinct vertices $v$, $x$ and $y$ in the
cycle 
such that $x\rightarrow v$ and $y\rightarrow v$ are both arcs in the cycle.
Since $x$ and $y$ are in  $\mathrm{desc}(u)$ we can find directed
paths $(u, \ldots, x, v)$ and $(u, \ldots, y, v)$ from $u$ to $v$.
The union of these two paths will contain a cycle and this cycle is
contained in $\mathrm{anc}(v)$.  Hence  $\mathrm{anc}(v)$ is not a tree and by vertex
transitivity   $\mathrm{anc}(u)$ is also not a tree.
The converse is proved similarly. 
\end{proof}

Given $a,b \in T$, the structure tree, we use $P(a,b)$ to
denote the set of vertices of $T$ belonging to the (unique) path in
$T$ from $a$ to $b$. Recall that the action of a group $G$ on a set $X$ is said to be \emph{primitive} if no non-trivial equivalence relation on $X$ is preserved by the action. 

\begin{thm} \label{desctree}
Let $D$ be a $2$-arc-transitive locally-finite connected digraph with more than
two ends and $G = \Aut(D)$.  Let $u \in VD$ and 
assume that $G_u$ acts primitively on  $D^+(u)$ (or on
$D^-(u)$).
Then $\desc(u)$ and $\anc(u)$ are trees.
\end{thm}

\begin{proof}   Let $e_0$ be a D-cut,  $E=Ge_0\cup Ge_0^*$ the
  associated tree set and $T=T(E)$ the structure tree.  

Since $D$ has more than two ends by Lemma~\ref{just2ends} we can
assume without loss of generality that there exists $r \geq 1$ such
that there are $a,b \in D^r(u)$ with $\phi(a) \neq \phi(b)$ (the other
possibility, with $r \leq -1$ is treated in the same way working with
$\mathrm{anc}(u)$ instead). Let $r \geq 1$ be the smallest integer
such that this is the case.    

\smallskip

\noindent {\bf Claim.} For distinct vertices $v$ and $w$ in $D^+(u)$ we have
\[
 D^{r-1}(v)  \cap  D^{r-1}(w)  = \varnothing.
\]

\begin{proof}[Proof of Claim]
By minimality of $r$ we see that if 
 $v$ is a vertex in $D$ then $\phi$ is constant on $D^{r-1}(v)$.   We
define an equivalence relation on $D^+(u)$ by saying that
vertices $v$ and $w$ in $D^+(u)$ are equivalent if and only if $\phi$
maps $D^{r-1}(v)$ and $D^{r-1}(w)$ to the same vertex in $T$.  This
equivalence relation is clearly  $G_u$ invariant.  Because $G_u$ acts
primitively on $D^+(u)$ we see that this equivalence relation either
has just one class or each class has just a single element.  The first
option is impossible by the choice of $r$ so each equivalence class has just a
single element.  If $v$ and $w$ are distinct vertices in $D^+(u)$ and 
$D^{r-1}(v)  \cap  D^{r-1}(w)$ is non-empty then 
$\phi$ would map $D^{r-1}(v)$ and
$D^{r-1}(w)$ to the same vertex in $T$ which is a contradiction.
\end{proof}

By Lemma~\ref{biglemma}(iv) there cannot be any directed cycles in
$D$.  The argument used in Lemma~\ref{decandanc} shows that if there is
a cycle in $\desc(u)$ then there is a cycle made up of two directed
paths $(u, x_1, \ldots, x_{k-1}, v)$ and $(u, y_1, \ldots, y_{k-1},
v)$ such that $x_i\neq y_i$ for $1=1,\ldots, k-1$.  We show that there
can not be any such cycle in $D$. 
  
There are two cases to consider. First suppose that $k+1 > r$.  
Then
$r \leq k$ and $x_r \in D^{r-1}(x_1)$ and $y_r \in D^{r-1}(y_1)$. We saw in the proof
of the claim above 
that the set of values $\phi$ takes on
$D^{r-1}(x_1)$ is disjoint from the set of values $\phi$ takes on
$D^{r-1}(y_1)$
and therefore $\phi(x_r) \neq \phi(y_r)$. By
Lemma~\ref{biglemma}\eqref{distances}, 
$d_T(\phi(u),\phi(v)) = (k+1)N$
and $d_T(\phi(u),\phi(x_r)) = rN = d_T(\phi(u),\phi(y_r))$. Since
$\phi(x_r)$ and $\phi(y_r)$ both lie on the path $P(\phi(u),\phi(v))$
in $T$ it follows that $\phi(x_r) = \phi(y_r)$ which is a
contradiction. 

Now we suppose that $k+1\leq r$.  Then $D^k(x_1)$ and $D^k(y_1)$ have
a common element $v$ and if $w\in D^{r-k}(u)$ then $w$ is in both
$D^{r-1}(x_1)$ and $D^{r-1}(y_1)$ contrary to the claim above.    

Thus it is impossible to find a cycle in $\desc(u)$ and $\desc(u)$ is a tree. Then, by Lemma~\ref{decandanc}, $\anc(u)$ is also a tree.   
\end{proof}

\begin{corol}\label{2trans}
Let $D$ be a $2$-arc-transitive locally-finite digraph with more than
two ends and $G=\Aut(D)$.  Let $u \in VD$ and assume that $G_u$ acts
doubly transitively  on 
$D^+(u)$ and $D^-(u)$.   Then $\desc(u)$ and $\anc(u)$ are trees.
\end{corol}

\begin{thm} \label{notuniversal2}
Let $D$ be a $2$-arc-transitive locally-finite digraph with more than
one end, let $G=\Aut(D)$ and $v \in VD$. If $G_v$ acts primitively 
both on $D^+(v)$ and $D^-(v)$ then the reachability relation
$\mathcal{A}$ is not universal and hence $\Delta(D)$ is bipartite.
\end{thm}

\begin{proof}   Let $e_0$ be a D-cut,  $E=Ge_0\cup Ge_0^*$ the
  associated tree set and $T=T(E)$ the structure tree. 

We consider various cases depending on the behavior of the structure
mapping $\phi$. 

By Lemma~\ref{biglemma}(ii), 
for each pair of adjacent vertices in $D$ the distance between their
images in $T$ is always some constant $N$.  Let $r$ be an  integer
$0\leq r\leq N$.  For a vertex $x\in D^+(v)$ we define $\phi_r(x)$ as
the unique 
vertex on the path $P(\phi(v), \phi(x))$ that is in distance $r$ from
$\phi(v)$ in $T$.  The fibers of $\phi_r$ define a $G_v$ invariant
equivalence relation on $D^+(v)$.  Since $G_v$ acts primitively on
$D^+(v)$ we know that either $\phi_r$ is a constant map or $\phi_r$ is
injective. Let $r$ denote the largest integer $0\leq r\leq N$ such
that $\phi_r$ is constant and define $c_v^+$ as the vertex in $T$ that
is in the image of $\phi_r$.   Thus if $x, y$ are distinct vertices in
$D^+(v)$ then the paths in $T$ from $\phi(v)$ both pass through
$c_v^+$ but then part their ways.
Similarly we define $c_v^-$ by
considering $D^-(v)$ instead of $D^+(v)$.

Moreover let $x \in D^+(v)$ and $x' \in D^-(v)$ and define the
following numbers 
\[
\begin{array}{lll}
F_1 = d_T(\phi(v),c_v^+), & \quad & F_2 = d_T(c_v^+,\phi(x)) \\
B_1 = d_T(\phi(v), c_v^-), & \quad & B_2 = d_T(c_v^-, \phi(x')).
\end{array}
\]
From the claim, and by vertex transitivity, it follows that these
numbers depend only on $D$. Note also that $B_1+B_2=N=F_1+F_2$.

By Lemma~\ref{biglemma}\eqref{Seifter}, given any $x \in D^-(v)$ and
$y \in D^+(v)$ the images $\phi(x)$ and $\phi(y)$ must be in different
connected components of $T \setminus \{ \phi(u) \}$.                 

Given $x,y \in D^+(v)$ and $x',y' \in D^-(v)$ with $x \neq y$ and $x'
\neq y'$ the graph induced by $P(\phi(x'),\phi(x)) \cup
P(\phi(y'),\phi(y))$ is depicted below. 

\begin{center}
\resizebox{.55\textwidth}{!}{\includegraphics[angle=270]{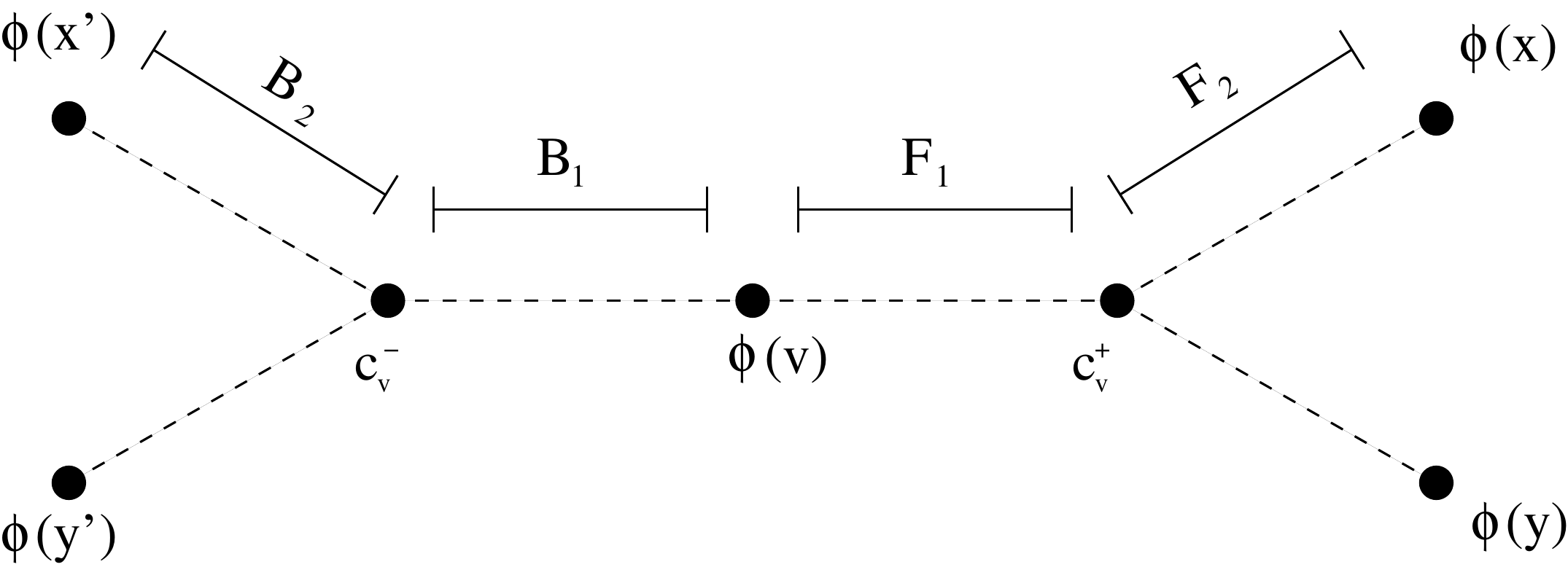}}
\end{center}

Either $F_2=0$, in which case
$\phi(x)=\phi(y)$ for all $x,y \in D^+(v)$, or $F_2 >0$, in which case
$\phi(x) \neq \phi(y)$ for all $x,y \in D^+(v)$ with $x \neq y$. The
corresponding statements for $B_2$ and $D^-(v)$ also hold. 

Suppose that $F_2=0$.  Then $\phi$ is not injective on $D^+(v)$, and it
may or may not be injective on $D^-(v)$. Let $W = (x_1, y_1, x_2, y_2,
\ldots)$ be an alternating walk in $D$. If $x_1 \rightarrow
y_1$ then for all $i$ since $y_i, y_{i+1} \in D^+(x_{i+1})$ we have
$\phi(y_i) = \phi(y_{i+1})$ and thus $\phi(y_i) = \phi(y_1)$. On the other hand, if the
walk $W$ starts with the arc $x_1\leftarrow y_1$ then we find that
$\phi(x_i)=\phi(x_1)$ for all $i$.  In either case it
follows that the image of $W$ under $\phi$ is in a bounded distance from
$\phi(y_1)$ and so $\mathcal{A}$ cannot be universal by
Lemma~\ref{biglemma}\eqref{distances}. Likewise we can deal with the
case $B_2=0$, so from now on we may assume $F_2>0$ and $B_2>0$
(i.e. $\phi$ is injective both on $D^+(v)$ and $D^-(v)$). 

We have two cases to consider, depending on the values of $F_2$ and $B_1$.

\bigskip

\begin{case}
$F_2 > B_1$.
\end{case}

\bigskip

Suppose $(u,v,w)$ is a directed path.  If the reachability relation
was universal then there would be an alternating path $W = (x_1, y_1,
x_2, y_2, \ldots, x_s, y_s)$  or $W = (x_1, y_1,
x_2, y_2, \ldots, x_s)$ such that the first arc traversed would
be $u\rightarrow v$ and the last one would be $v\rightarrow w$. We show
this to be impossible by demonstrating that $\phi(x_1)\neq \phi(x_s)$
and $\phi(x_1)\neq\phi(y_s)$ (for all $s>1$).   

Since $F_2 > B_1$, we know that 
  $c_{x_i}^+$ lies on the path 
$P(\phi(x_i), c_{y_i}^-)$ and $ c_{y_i}^-$ lies on the path
  $P(c_{x_i}^+, \phi(y_i))$.
The path $P(\phi(x_i), \phi(x_{i+1}))$ 
can be split up into three segments, the first one
being the path  between $\phi(x_i)$ and $c_{x_i}^+$ (which has
length $F_1$), then the
path between $c_{x_i}^+$ and $c_{y_i}^-$ (which has length $F_2-B_1$)
and finally the path between $c_{y_i}^-$ and $\phi(x_{i+1})$ (which
  has length $B_2$).      
In general the  paths $P(c_{x_i}^+, c_{y_i}^-)$ and 
$P(c_{y_i}^-, c_{x_{i+1}}^+)$
have only the vertex $c_{y_i}^-$ in
common.  We also note that $P(c_{y_i}^-, c_{x_{i+1}}^+)$ 
and $P(c_{x_{i+1}}^+, c_{y_{i+1}}^-)$ have only the vertex
$c_{x_{i+1}}^+$ in common.  The paths $P(c_{x_i}^+, c_{y_i}^-)$ 
and $P(c_{y_i}^-, c_{x_{i+1}}^+)$ both have
length $F_2-B_1$.   By joining together all the
paths $P(c_{x_i}^+, c_{y_i}^-)$ and 
$P(c_{y_i}^-, c_{x_{i+1}}^+)$ for $i=1$ to $i=s-1$ we get a path $Q$ in
$T$ from $c_{x_1}^+$ to $c_{x_s}^+$ and this path has length
$2(s-1)(F_2-B_1)$.   
Then $P(\phi(x_1), c_{x_1}^+)\cup Q\cup P(c_{x_{s+1}}^+, \phi(x_s))$
is a path in $T$ from $\phi(x_1)$ to $\phi(x_s)$ and has strictly positive
length if $s>1$.  Hence $\phi(x_1)\neq\phi(x_s)$.  That 
$\phi(x_1)\neq\phi(y_s)$ is proved similarly.

\bigskip

\begin{case}
$F_2 \leq B_1$ (and so $B_2 \leq F_1$).
\end{case}

\bigskip

Here we see that both
  $c_{x_1}^+$ and $c_{y_1}^-$ lie on the path 
$P(\phi(x_1), \phi(y_1))$.  Since we
  are assuming that $F_2\leq B_1$ we know that $c_{x_1}^+$ lies
  between $c_{y_1}^-$ and $\phi(y_1)$.  Since $c_{y_1}^-$ is also on
  the path $P(\phi(y_1), \phi(x_2))$ we see that $c_{x_1}^+$ is also contained
  in the path $P(\phi(y_1), \phi(x_2))$.  Both $c_{x_1}^+$ and 
 $c_{x_2}^+$ are in distance $F_2$ from $\phi(y_1)$ and  
both lie on the  path  $P(\phi(y_1), \phi(x_{2}))$.  This implies that
  $c_{x_1}^+=c_{x_2}^+$.  Then of course we get by induction that
$c_{x_1}^+=c_{x_2}^+=\cdots =c_{x_{s}}^+$.  Whence $d(c_{x_1}^+,
  \phi(x_i))=F_1$ and  $d(c_{x_1}^+, \phi(y_i))=F_2$ for every $i$.  From this we
  get a uniform bound on the diameter of the image under $\phi$ of
  every alternating path in $D$.  Hence it is impossible that the
  relation $\mathcal{A}$ is universal.
\end{proof}

\begin{corol} \label{notuniversal2_corol}
Let $D$ be a $2$-arc-transitive locally-finite
digraph with more than one end.  Set $G=\aut(D)$.  If for a vertex $v\in
VD$ the group $G_v$
acts doubly transitively 
both on $D^+(v)$ and $D^-(v)$ then the reachability relation
$\mathcal{A}$ is not universal and hence $\Delta(D)$ is bipartite.
\end{corol}

A digraph $D$ is said to have property $Z$ if there is a digraph homomorphism from $D$ onto the two-way infinite directed line.  The example
given in \cite{Malnic1} shows that the conclusion of the
theorem above cannot be replaced by ``has property $Z$''; see Section~5 below for more on this.  Of course
not every highly arc-transitive digraph satisfies the hypotheses of
the above theorem (counterexamples are constructed easily using the
universal covering construction of \cite{Cameron2}) so
Theorem~\ref{notuniversal2} 
cannot be used to prove in general that $\mathcal{A}$ is not universal for
any locally-finite highly arc-transitive digraph. 

\section{Connected-homogeneous digraphs without triangles}

We now turn our attention to the main subject of this article: the study of connected-homogeneous digraphs. Recall
from the introduction that a digraph $D$ is called \emph{connected-homogeneous} (or simply $C$-homogeneous) if any
isomorphism between finite connected induced subdigraphs of $D$
extends to an automorphism of $D$. We begin here by considering the case where $D$ has more than one end
and is triangle-free, meaning that the underlying graph of $D$ does not embed a
triangle. By applying results from the previous section we shall see that these digraphs form a special subfamily of the infinite highly arc-transitive digraphs. Note that
there are examples of one-ended $C$-homogeneous locally finite
digraphs. For example it may be verified that
the one-ended digraph constructed in \cite[Example~1]{Moller6} is
$C$-homogeneous. 

Before stating the main result of this section we first need a construction which will be used to 
build examples of $C$-homogeneous digraphs. This construction was introduced in \cite{Cameron2} 
where it was used to construct universal covering digraphs for highly arc-transitive-digraphs.

Let $\Delta$ be an edge-transitive, connected bipartite graph with given bipartition $X \cup Y$. Let $u = |X|$ and $v = |Y|$, noting that in general $u$ and $v$ need not be finite. We shall construct a digraph $DL(\Delta)$ that has the property that its reachability graph is isomorphic to $\Delta$. Let $T$ be a directed tree with constant in-valency $u$ and constant out-valency $v$. For each vertex $t \in T$ let $\varphi_t$ be a bijection from $T^{-}(t)$ to $X$, and let $\psi_t$ be a bijection from $T^+(t)$ to $Y$. Then $DL(\Delta)$ is defined to be the digraph with vertex set $ET$ such that for $(a,b), (c,d) \in ET$, $((a,b),(c,d))$ is a directed edge of $DL(\Delta)$ if and only if $b=c$ and $(\psi_b(a),\varphi_b(d))$ is an edge of $\Delta$.  The digraph $DL(\Delta)$ may be thought of as being constructed by taking $T$ and replacing each vertex of $T$ by a copy of $\Delta$. Then for copies of $\Delta$ that are indexed by adjacent vertices $a$ and $b$ of $T$ we identify a single vertex from one of the copies of $\Delta$ with a vertex from the other copy of $\Delta$, with the bijections determining the identifications. Since $\Delta$ is edge transitive it follows that different choices of bijections $\varphi_y$ and $\psi_y$ for $y \in VT$ will lead to isomorphic digraphs, and $DL(\Delta)$ is used to denote this digraph. 

Note that without the assumption that $\Delta$ is edge transitive, different choices of $\varphi_y$ and $\psi_y$ for $y \in VT$ can lead to non-isomorphic digraphs. For instance, suppose that $\Delta$ is the bipartite graph where $X = \{ x,x'\}$, $Y = \{y,y'\}$ and the edges are $(x,y)$, $(x,y')$ and $(x',y')$. Then $\Delta$ is clearly not edge transitive. Now the functions $\varphi_y$ and $\psi_y$ for $y \in VT$ can be defined in such a way that every vertex in the digraph $DL(\Delta)$ arising from the above construction will have an even number of directed edges adjacent to it 
(i.e. each vertex either has in- and out-degree $2$, or in- and out-degree $1$). 
But different choices of functions $\varphi_y$ and $\psi_y$ for $y \in VT$ can give rise to vertices in $DL(\Delta)$ with exactly three edges adjacent to them, and hence to a different digraph. 

It is immediate from the definition of $DL(\Delta)$ that the reachability graph of $DL(\Delta)$ is $\Delta$. In other words, we have $\Delta(DL(\Delta)) = \Delta$.  

\begin{thm}\label{bigtheorem}
Let $D$ be a connected locally-finite triangle-free digraph with more
than one end. If $D$ is $C$-homogeneous then
\begin{enumerate}[(i)]  
\item $D$ is highly arc-transitive, 
\item $\mathcal{A}$ is not universal and hence $\Delta(D)$ is bipartite,
\item in particular $\Delta(D)$ is isomorphic to one of: a cycle
  $C_m$ ($m$ even), a complete bipartite graph $K_{m,n}$ ($m,n \in
  \mathbb{N}$), complement of a perfect matching $CP_n$ ($n \in
  \mathbb{N}$), or an infinite
  semiregular tree $T_{a,b}$ ($a,b \in \mathbb{N}$),
\item conversely for every bipartite graph $B$ listed in (iii), the digraph $DL(B)$ is a connected triangle-free locally-finite C-homogeneous digraph with $\Delta(DL(B))$ isomorphic to $B$.   
\end{enumerate}
In addition, if $D$ has more than two ends then
\begin{enumerate}[(i)]
\item[(v)] $\mathrm{desc}(u)$ (and
  $\anc(u)$) is a tree,
  for all $u \in VD$.    
\end{enumerate}
\end{thm}
\begin{proof}
For part (i), observe that from $C$-homogeneity and the absence of triangles 
it follows that $D$ is $2$-arc
transitive and that, with $G = \Aut(D)$, $G_v$ acts doubly transitively on $D^+(v)$ and
$D^-(v)$ for any vertex $v$. Applying Lemma~\ref{biglemma}\eqref{nocycles} we see that the subdigraph induced by a $k$-arc $(v_0, \ldots, v_k)$ only
contains the arcs $(v_i, v_{i+1})$. It then follows from
$C$-homogeneity that given any other $k$-arc $(w_0, \ldots, w_k)$
there is an automorphism $\alpha$ with $\alpha(v_i) = w_i$ for $0 \leq
i \leq k$. Thus $D$ is $k$-arc-transitive for all $k$, and therefore
highly arc-transitive. Part (v) follows from Corollary~\ref{2trans} 
and (ii) follows from Corollary~\ref{notuniversal2_corol}. 

Part (iii) follows from Theorem~\ref{chomogeneousbipartite} below. 

For part (iv), the facts that $DL(B)$ is connected, triangle-free, locally-finite and that $\Delta(DL(B)) \cong B$ are all immediate from the definition of $DL(B)$. That $DL(B)$ is $C$-homogeneous will be proved in Theorem~\ref{CPWalwaysworks} below. \end{proof}

In the next section we give some additional examples, showing that the construction $DL(B)$ on its own does not exhaust all examples of infinitely ended triangle-tree locally-finite $C$-homogeneous digraphs. In Section~6 we shall see that Theorem~\ref{bigtheorem}(v) does not hold in the $2$-ended case. The rest of this section will be devoted to proving Theorems~\ref{chomogeneousbipartite} and \ref{CPWalwaysworks} which will establish parts (iii) and (iv) of Theorem~\ref{bigtheorem}. 

We begin by quickly dealing with the situation where the out-degree (or dually the in-degree) of $D$ is equal to $1$. Let $D$ be a connected locally-finite triangle-free $C$-homogeneous digraph with more than one end. If $D$ has out-degree equal to $1$ then the only cycles that $D$ can contain are directed cycles. But by Lemma~\ref{biglemma}(iv), $D$ does not have any directed cycles. It follows that $D$ is a directed tree. 
Dually if $D$ had in-degree $1$ then $D$ would be a tree. 
Therefore, since $C$-homogeneity implies vertex transitivity, if the indegree of $D$ equals $1$ then 
$D$ is isomorphic to the unique directed tree where every vertex has indegree $1$ and outdegree $k$ for some fixed $k \in \mathbb{N}$. There is an obvious dual statement if the outdegree of $D$ is $1$. Hence, since our interest is in classification, from now on we may, in the triangle-free case, 
suppose that 
$D$ has in-degree at least $2$ and out-degree at least $2$. 

We know from Theorem~\ref{bigtheorem}(ii) that $\Delta(D)$ is a
bipartite graph. Now we want to determine all the possibilities for
$\Delta(D)$. To do this we shall show that the bipartite graph $\Delta(D)$ inherits $C$-homogeneity from $D$, so that $\Delta(D)$ is a locally-finite $C$-homogeneous bipartite graph.

\begin{defn}\label{hom_bipartite}
Let $\Gamma$ be a bipartite graph with given bipartition $X \cup Y$. By an \emph{automorphism of the bipartite graph} $\Gamma = X \cup Y$ we mean a bijection $\varphi : \Gamma \rightarrow \Gamma$ that sends edges to edges, non-edges to non-edges, and preserves the bipartition (i.e. $\varphi(X) = X$ and $\varphi(Y) = Y$). We say that $\Gamma$ is a \emph{homogeneous bipartite graph} if every 
isomorphism $\theta : \lb X' \cup Y' \rb \rightarrow \lb X'' \cup Y'' \rb$ (where $X'$, $X''$ are subsets of $X$, and $Y'$, $Y''$ are subsets of $Y$) between finite induced subgraphs of $\Gamma$ that preserves the bipartition (i.e. $\theta(X') = X''$ and $\theta(Y')=Y''$) extends to an automorphism of $\Gamma$. We say that $\Gamma$ is a  \emph{$C$-homogeneous bipartite graph} if any isomorphism between finite induced connected subgraphs which preserves the bipartition extends to an automorphism of the bipartite graph $\Gamma$.
\end{defn}

The countable homogeneous bipartite graphs were classified in \cite{Goldstern1}. We shall apply this result below in Lemma~\ref{neighbourhoods} where we list the finite homogeneous bipartite graphs. Note that there is a difference between homogeneous bipartite graphs and homogeneous graphs that are bipartite (for example, $K_{2,3}$ is a homogeneous bipartite graph, but as a graph it is not homogeneous since it is not even vertex transitive).  

\begin{lem}\label{easyone}
Let $D$ be a triangle-free locally-finite infinite connected $C$-homogeneous digraph with more than one end. Then the underlying undirected graph of $\Delta(D)$ is a locally-finite $C$-homogeneous bipartite graph.
\end{lem}
\begin{proof}
By Theorem~\ref{notuniversal2}, $\Delta(D)$  is bipartite, and it clearly inherits local-finiteness from $D$. To see that $\Delta(D)$ is $C$-homogeneous, fix a copy $\lb \mathcal{A}(e) \rb$ $(e \in ED)$ of $\Delta(D)$ in $D$. Clearly $\lb \mathcal{A}(e) \rb$ contains at least one arc, namely the arc $e \in ED$. Since $D$ is $C$-homogeneous any isomorphism between finite connected induced substructures of $\lb \mathcal{A}(e) \rb$ which preserves the bipartition extends to an automorphism of $D$. This automorphism of $D$ preserves the $\mathcal{A}$ equivalence classes of edges given by the reachability relation, hence it fixes $\lb \mathcal{A}(e) \rb$ setwise. Thus by restricting to $\lb \mathcal{A}(e) \rb$ we get an automorphism of this bipartite graph  which extends the original isomorphism between connected subgraphs.
\end{proof}

This reduces the problem of determining the possibilities of the reachability graph $\Delta(D)$ to the 
problem of classifying the 
locally-finite $C$-homogeneous bipartite graphs.

\subsection*{\boldmath Classifying the locally-finite $C$-homogeneous bipartite graphs}

The finite and locally-finite $C$-homogeneous graphs were classified in \cite{Gardiner2} and \cite{Enomoto1}, and the countably infinite $C$-homogeneous graphs were classified in \cite{Gray1}. Note that as for homogeneity, 
there is a difference between a $C$-homogeneous bipartite graph and a $C$-homogeneous graph that happens to be bipartite. Therefore in order to classify the $C$-homogeneous bipartite graphs it is not simply a case of reading off those undirected $C$-homogeneous graphs that happen to be bipartite. It turns out, however, that many of the arguments used in \cite{Gray1} for the study of $C$-homogeneous graphs may be easily adapted in order to obtain results about $C$-homogeneous bipartite graphs.  We now give a classification of 
the locally-finite $C$-homogeneous bipartite graphs, which is stated in Theorem~\ref{chomogeneousbipartite} below. We note that extending this to all countable bipartite graphs, not just those that are locally-finite, would not require much more work, with the only additional examples being infinite valency analogues of locally-finite ones, and the generic `random' homogeneous bipartite graph (i.e. the Fra{\"{\i}}ss{\'e} limit of the family of all finite bipartite graphs; see \cite{evans} for details).  

For the remainder of this subsection $\Gamma$ will denote a connected $C$-homogeneous locally-finite bipartite graph with bipartition $X \cup Y$. If the vertices of $X$ (or of $Y$) all have degree $1$ then $\Gamma \cong K_{1,m}$ for some $m \geq 1$, so we shall suppose that this is not the case.

The following lemma is proved by modifying the arguments of 
\cite[Lemma~7, Lemma~30]{Gray1}.

\begin{lem}\label{embedsasquare}
Let $\Gamma$ be a connected $C$-homogeneous locally-finite bipartite graph with bipartition $X \cup Y$. 
If $\Gamma$  is not a tree and has at least one vertex with degree greater than $2$ then $\Gamma$  
embeds $C_4$ as an induced subgraph. 
\end{lem}
\begin{proof}
Let $C_n$ be the smallest cycle that embeds into $\Gamma$. First we establish $n \leq 6$ and
then we shall rule out $n=6$ as a possibility. 

Suppose, seeking a contradiction, that the smallest cycle that embeds is $C_n$ where $n \geq 7$. Note that $n$ must be even since $\Gamma$ is bipartite. Fix a 
vertex $v \in V\Gamma$ with degree at least $3$. Such a vertex exists by the assumptions of the lemma. 
Let $a,b,c$ be distinct elements of $\Gamma(v)$; they will be pairwise non-adjacent as $\Gamma$ is bipartite. By $C$-homogeneity
 the path $(b,v,c)$ extends to a copy $(v,b,b_1, \ldots, b_k, c)$ of $C_n$. Since $n \geq 7$ it follows that $k \geq 4$ and 
therefore $a$ is not adjacent to $b_i$ for all $1 \leq i \leq k$ (since any such edge would create a cycle in
 $\Gamma$ shorter than $C_n$ itself). Also, $a$ is not adjacent to $c$ or $b$ since they all belong to the same part of the bipartition of $\Gamma$. Now by $C$-homogeneity there is an automorphism $\alpha$, preserving the bipartition,  and satisfying
$$
\alpha(\lb c,v,b,b_1,b_2,\ldots,b_{k-1} \rb) = \lb a,v,b,b_1,b_2,\ldots,b_{k-1} \rb.
$$   
Note that the vertex $\alpha(b_k)$ does not belong to $B = \{ v,a,b,c,b_1,b_2, \ldots, b_k \}$ since $\alpha(b_k)$ is adjacent both to $a$ and to $b_{k-1}$, and none of the vertices in $B$ have this property. Consider $\lb a,v,c,b_k,b_{k-1},\alpha(b_k)\rb$. If $\alpha(b_k)$ is adjacent to $c$ (which
happens for instance if $\alpha(b_k)=b_k$) then  
$\lb a,v,c,\alpha(b_k) \rb \cong C_4$ which is a contradiction. Otherwise $\alpha(b_k) \not\sim c$ and $\alpha(b_k) \not\sim b_k$ (since they are in the same part of the bipartition) and $\lb a,v,c,b_k,b_{k-1},\alpha(b_k) \rb \cong C_6$, which is again a contradiction. We conclude that $n \leq 6$, and hence since $\Gamma$ is bipartite, the only possibilities are $n=4$ or $n=6$. So to complete the proof of the lemma it suffices to show that $n \neq 6$. 

Seeking a contradiction suppose that $C_6$ is the smallest cycle which embeds in $\Gamma$.
Fix an edge $\{x,y\}$ in the graph $\Gamma = X \cup Y$ with $x \in X$ and $y \in Y$, and without loss of generality suppose that 
$x$ has degree at least $3$. 
Let $\{x_i : i \in I\} = \Gamma(x) \setminus \{y\}$ and let $\{y_j : j \in J\} 
= \Gamma(y) \setminus \{x\}$. Also define $X_i = \Gamma(x_i) \setminus \{x\}$ for $i \in I$, and $Y_j = \Gamma(y_j) \setminus \{y\}$ 
for each $j \in J$. Our assumptions on cycles ensure that distinct sets $X_i$ and $X_{i'}$ are disjoint with no edges between them for all $i, i' \in I$, and distinct sets $Y_j$ and $Y_{j'}$ are disjoint with no edges between then for all $j,j' \in J$. Moreover the sets $X_i$ and $Y_j$ are disjoint for all $i \in I$, $j \in J$, since $X_i \subseteq X$ while $Y_i \subseteq Y$. 

Let $i \in I$, $j \in J$ and consider the graph $\lb X_i \cup Y_j \rb$. For every $v \in Y_j$ there is at most one $u \in X_i$ with $v \sim u$, for otherwise with $u, u' \in X_i$, $v \sim u$ and $v \sim u'$ then $\langle x_i, u, u', v \rangle$ would be a square, a contradiction. Similarly, every $u \in X_i$ is adjacent to at most one $v \in Y_j$. Since $\Gamma$ embeds $C_6$ there is at least one pair $u \in X_i$ and $v \in Y_j$ with $u \sim v$.
By $C$-homogeneity the pointwise stabiliser of $\{ x,y,x_i,y_j \}$ acts transitively on $Y_j$ and also acts transitively on $X_i$. It follows that every vertex in $X_i$ is adjacent to exactly one vertex of $Y_j$, and vice versa. Hence $\lb X_i \cup Y_j \rb$ is a perfect matching. In particular $|X_i| = |Y_j|$ for all $i \in I$ and $j \in J$, and this implies $|J| = |X_i| = |Y_j| = |I|$. It follows from this that all of the vertices in $\Gamma$ have the same degree, and that this degree is strictly greater than $2$.   

Let $y_1 \in \Gamma(y) \setminus \{ x \}$ and $x_1, x_2 \in \Gamma(x) \setminus \{ y \}$ with $x_1 \neq x_2$. This is possible since $x$ has degree at least $3$. There
 is a map $\varphi:X_1 \rightarrow X_2$ given by composing the bijection from $X_1$ to $Y_1$, and that from 
$Y_1$ to $X_2$, given by the perfect matchings $\langle X_1 \cup Y_1 \rangle$ and $\langle X_2 \cup Y_1 \rangle$. Fix $a,b \in X_1$ with $a \neq b$. By $C$-homogeneity there is an 
automorphism $\alpha$ fixing $\{ x,y,x_1,x_2,y_1,a,b \}$ and interchanging $\varphi(a)$ and $\varphi(b)$.
Then the unique neighbour of $a$ in $Y_1$ is adjacent to $\varphi(a)$ but not to $\varphi(b)$, a contradiction.  
\end{proof}

\begin{lem}\label{neighbourhoods}
Let $\Gamma$ be a connected $C$-homogeneous locally-finite bipartite graph with bipartition $X \cup Y$.
Let $x \in X$ and $y \in Y$ be such that $\{x,y\} \in E\Gamma$, and define $A = \Gamma(x) \setminus \{y\}$ and $B = \Gamma(y) \setminus \{ x\}$. Then $\Omega = \lb A \cup B \rb$ is a finite homogeneous bipartite graph, and therefore is one of: a null bipartite graph, complete bipartite, complement of a perfect matching, or a perfect matching.
\end{lem}
\begin{proof}
If either $A$ or $B$ is empty then $\Omega$ is a null bipartite graph, 
so suppose otherwise. 
The graph $\Omega$ is finite since $\Gamma$ is locally finite. Let $\varphi:U \rightarrow V$ be an isomorphism between induced subgraphs of $\Omega$ that preserves the bipartition. 
Then extend this isomorphism to $\hat{\varphi}: \{ x,y\} \cup U \rightarrow \{x,y\} \cup V$ by defining $\hat{\varphi}(x) = x$ and $\hat{\varphi}(y) = y$. Now $\hat{\varphi}$ is an isomorphism between connected induced subgraphs and so by $C$-homogeneity extends to an automorphism $\alpha$ of $\Gamma$. But $g$ fixes both $x$ and $y$, so $\alpha(\Omega) = \Omega$ and the restriction of $\alpha$ to $\Omega$ is an automorphism of the bipartite graph $\Omega$ extending $\varphi$. 
Finally, by inspection of the list of homogeneous bipartite graphs given in \cite[Section~1]{Goldstern1} we obtain the possibilities for $\Omega$ listed in the lemma.
\end{proof}

Note that by Lemma~\ref{embedsasquare} if $\Gamma$ is not a cycle or a tree then the bipartite graph $\Omega$ in Lemma~\ref{neighbourhoods} has at least one edge, and so is not a null bipartite graph. We are now in a position to complete the classification of locally-finite $C$-homogeneous bipartite graphs, which will prove part (iii) of Theorem~\ref{bigtheorem}.    

\begin{thm}\label{chomogeneousbipartite}
A connected graph $\Gamma$ is a locally-finite $C$-homogeneous bipartite graph if and only if $\Gamma$ is isomorphic to one of the following:
\begin{enumerate}
\item cycle $C_m$ ($m$ even);
\item infinite semiregular tree $T_{a,b}$ ($a,b \in \mathbb{N}$);
\item complete bipartite graph $K_{m,n}$ ($m,n \in \mathbb{N}$);
\item complement of a perfect matching.
\end{enumerate}
\end{thm}
\begin{proof}
Clearly each of the graphs listed is a locally-finite $C$-homogeneous bipartite graph. For the converse, let $\Gamma = X \cup Y$ be an arbitrary locally-finite $C$-homogeneous bipartite graph. By definition each vertex in $\Gamma$ has one of two possible valencies (depending on the part of the bipartition that the vertex belongs to). If all the vertices in $X$ (or dually in $Y$) have degree $1$ then $\Gamma \cong K_{m,1}$ for some $m$, so suppose otherwise. Also, if $\Gamma$ is a cycle or a tree then we are done, so suppose not. Thus $\Gamma$ satisfies the hypotheses of Lemma~\ref{embedsasquare} and hence embeds a square $C_4$.   

Let $x$, $y$, $A$, $B$ and $\Omega$ be as defined in the statement of Lemma~\ref{neighbourhoods}, noting that $\Omega$ must have at least one edge since $\Gamma$ embeds a square. We now consider each possibility for $\Omega$, as listed in Lemma~\ref{neighbourhoods}, determining the possibilities for $\Gamma$ in each case.

\bigskip

\setcounter{case}{0}

\begin{case} $\Omega$ is complete bipartite. \end{case}
Then $(\{x\} \cup B) \cup (\{y\} \cup A)$ induces a complete bipartite graph. Also since $x$ has degree $|A|+1$, every vertex in $B$ also has degree $|A|+1$. Similarly every vertex in $A$ has degree $|B|+1$. Since $\Gamma$ is connected it follows that there are no other vertices, i.e. $\Gamma =    (\{x\} \cup B) \cup (\{y\} \cup A)$ and so $\Gamma$ is a complete bipartite graph.

\bigskip

\begin{case} $\Omega$ is the complement of perfect matching with at least four vertices. \end{case}
Since the degree of $x$ is $|A|+1$ it follows that every vertex in $B$ also has
degree $|A|+1$. Let $b \in B$ and let $z$ be the unique neighbour of
$b$ not in $A \cup \{y\}$. Fix $a \in A$ with $a\sim b$. By
$C$-homogeneity $|\Gamma(y) \cap \Gamma(a)| = |\Gamma(y) \cap
\Gamma(z)| = |B|$. Since $z \not\sim x$ it follows that $B \subseteq
\Gamma(z)$. Let $z'$ be the unique
neighbour of $z$ not in $B \cup \{x\}$. Since $|\Gamma(x) \cap \Gamma(b)|
\geq 2$ (because squares embed by Lemma~\ref{embedsasquare}) it
follows that $|\Gamma(b) \cap \Gamma(z')| \geq 2$, which in turn
implies that there exists $v \in A$ with $v \sim z'$. Now using a dual
argument to the one above we conclude that $\Gamma(z') = A \cup
\{z\}$. This completely determines the structure of $\Gamma$, and we
conclude that $\Gamma$ is the complement of a perfect matching. 

\bigskip

\begin{case} $\Omega$ is a perfect matching. \end{case}
If $\Omega$ has either $2$ or $4$ vertices then this comes under one of the previous cases, so suppose that $\Omega$ has at least $6$ vertices. Following \cite[Proposition~33]{Gray1}
let $\{p_1, p_2, p_3\} \subseteq A$ and let $f: A \rightarrow B$ be the bijection determined by the perfect matching $\Omega$. Since the $2$-arcs $x y f(p_1)$ and $y x p_1$ both extend uniquely to squares it follows by $C$-homogeneity that the same is true for all $2$-arcs in the graph $\Omega$. For $i,j \in \{1,2,3\}$ let $r_{ij}$ be the vertex that extends the $2$-arc $p_i x p_j$ to a square. Clearly $r_{ij} \not\in \{x\} \cup B$ for any $1\leq i<j\leq 3$, and also $r_{12} \neq r_{23}$. By $C$-homogeneity there is an automorphism $\alpha$ satisfying $\alpha(r_{12}, p_1, x, p_3, r_{23}) = (r_{12}, p_1, x, p_3, f(p_3))$ (these connected substructures are each isomorphic to a line with $5$ vertices, since $\Gamma$ is bipartite). However, since $\alpha$ fixes pointwise the triple of vertices $\{ x, p_1, r_{12} \}$ it must also fix $p_2$. But this is impossible since $p_2 \sim r_{23}$ while $p_2 \not\sim f(p_3)$. This is a contradiction, and we conclude that this case (with $|\Omega| \geq 6$) does not happen.

This covers all possibilities for $\Omega$ and completes the proof of the theorem.
\end{proof}

Using the $DL(\Delta)$ construction described above, we 
now show that each of the bipartite graphs listed in Theorem~\ref{chomogeneousbipartite}
arises as the reachability graph of some
locally-finite $C$-homogeneous bipartite graph. This will complete the proof of part (iv) of Theorem~\ref{bigtheorem}.  

\begin{thm}\label{CPWalwaysworks}
The digraph $DL(\Delta)$ is $C$-homogeneous if and only if $\Delta$ is isomorphic to one of: a cycle $C_m$ ($m$ even), an infinite semiregular tree $T_{a,b}$ ($a,b \in \mathbb{N}$), a complete bipartite graph $K_{m,n}$ ($m,n \in \mathbb{N}$), or the complement of a perfect matching.
\end{thm}
\begin{proof}
We follow the same notation as used in the proof of \cite[Theorem~2.2]{Cameron2}. By Theorem~\ref{chomogeneousbipartite} it is sufficient to prove that $DL(\Delta)$ is $C$-homogeneous if and only if $\Delta$ is a $C$-homogeneous bipartite graph. 

If $DL(\Delta)$ is $C$-homogeneous then, by Lemma~\ref{easyone}, so is $\Delta$. For the converse let $\Delta$ be a connected $C$-homogeneous bipartite graph. Let $D_i$ ($i = 1,2$) be connected finite subdigraphs of $D=DL(\Delta)$ with $\varphi: D_1 \rightarrow D_2$ a given isomorphism. 
From the definition of $DL(\Delta)$ it is immediate that $\Delta(D) = \Delta(DL(\Delta)) \cong \Delta$.  
For an edge $e$ in $DL(\Delta)$ let $V(e)$ denote the vertex set of the bipartite graph $\lb \mathcal{A}(e) \rb$. We must extend $\varphi$ to an automorphism $\hat{\varphi}$ of $D$. Let $ED_1 = \{a_1, \ldots, a_r\} \subseteq ED$, and let $ED_2 = \{b_1, \ldots, b_r  \} \subseteq ED$  with $\varphi(a_i) = b_i$ for all $i$. Note that here we have specified $D_1$ and $D_2$ by listing edges rather than listing vertices. 

Observe that for any edge $e \in ED$ and any two vertices $x,y \in V(e)$, any path in $D$ from $x$ to $y$ must be contained in $V(e)$. This is a consequence of the fact that the set of all blocks $V(f)$ $(f \in ED)$ carries the structure of a tree, and any distinct pair of blocks $V(f)$ and $V(f')$ intersect in at most one vertex. So any path from $x$ to $y$ not contained in $V(e)$ would have to leave $V(e)$ and then later re-enter $V(e)$ at a common vertex, contradicting the definition of path. 
Since $D_1$ is connected, it follows from this observation that for all $e \in ED$ if $D_1$ intersects $V(e)$ then $\lb D_1 \cap V(e) \rb$ is a connected subdigraph of $\lb V(e) \rb$.

Since $\Delta$ is $C$-homogeneous, and $\lb D_1 \cap V(f) \rb$ is empty or connected for all $f \in ED$, for each edge $e \in D_1$ we may fix an isomorphism $\theta(e, \varphi(e)) : V(e) \rightarrow V(\varphi(e))$ which extends $\varphi \upharpoonright_{V(e) \cap VD_1}$. We now define $\hat{\varphi} : VD \rightarrow VD$ inductively as follows.  Begin by defining $\hat{\varphi} : V(a_1) \rightarrow V(b_1)$ as $\hat{\varphi} = \theta(a_1, b_1)$. Let $V_k$ denote the set of all vertices of $D$ that lie in a block $V(f)$ at distance at most $k$ (in the underlying tree of blocks) from the block $V(a_1)$. 
Suppose that $\hat{\varphi}$ has been defined on $V_j$ $(j \leq k)$ and consider $V_{k+1}$. Let $V(a)$ be a block where $a =  (u,v)$ is a directed edge that intersects $V_k$, chosen so that either $u \in V_k$ or $v \in V_k$, but not both. Suppose that $v \in V_k$ and $u \not\in V_k$, the reverse is dealt with using a similar argument. Let $a' \in ED$ be an arc satisfying $V(a') \cap V(a) = \{v\}$. Now there are two possibilities. 

First suppose that $V(a) = V(a_i)$ for some $a_i \in ED_1$. In this case, since $D_1$ is connected we can choose a path $p$ from a vertex of the edge $a_1 \in ED_1$ to a vertex of the edge $a_i \in ED_1$. Now, since $u \not\in V_k$, $v \in V_k$, and $u$ and $v$ are adjacent, it follows that the block $V(a) = V(a_i)$ lies at distance $k+1$ from the block $V(a_1)$. Then since $V(a') \cap V(a)= \{ v \}$ and $V(a')$ lies at distance $k$ from $V(a_1)$ in the underlying tree of blocks, it follows that the path $p$ must intersect the block $V(a')$ and must contain the vertex $v$. As $p$ is a path in $D_1$ we conclude that $v \in D_1$.  
Now we 
define $\hat{\varphi}: V(a) \rightarrow V(\varphi(a_i))$ as $\hat{\varphi} = \theta(a_i,\varphi(a_i))$. Note that $\hat{\varphi}$ has now been defined twice on the vertex $v$, using $\theta(a',\varphi(a'))$ as well, but in both cases $v$ is sent to $\varphi(v)$ since $v \in VD_1$. So the map remains well-defined.

On the other hand, if $V(a) \neq V(a_i)$ for any $a_i \in ED_1$ 
then let $v' = \hat{\varphi}(v)$, which is defined since $v \in V_k$, and let $u'$ be any vertex in $D$ with $u' \rightarrow v'$. Then setting $b = (u',v')$ we define 
$\hat{\varphi}: V(a) \rightarrow V(b)$ to be any isomorphism $\alpha: V(a) \rightarrow V(b)$ that satisfies $\alpha(v) = \hat{\varphi}(v)$. Such an isomorphism $\alpha$ exists since $\Delta(D) \cong \Delta$ is $C$-homogeneous.

Since any two distinct blocks at distance $k$ from $V(a_1)$ are disjoint, $\hat{\varphi}$ is a well-defined mapping from $VD$ to itself. The mapping $\hat{\varphi}$ is onto since every vertex $w \in VD$ lies in a block $V(f)$ at some finite distance $s$ from $V(b_1)$, and so after stage $s$ the vertex $w$ will be in the image. The mapping $\hat{\varphi}$ is one-one since given $x,y \in VD$ with $x \neq y$ if there exists $f \in ED$ with $x,y \in V(f)$ then $\hat{\varphi}(x) \neq \hat{\varphi}(y)$ since $\hat{\varphi} \upharpoonright_{V(f)}$ is a bijection with range $V(f)$. Otherwise, there exist disjoint blocks $V(a)$ and $V(a')$ with $x \in V(a)$ and $y \in V(a')$, and $\hat{\varphi}(x) \neq \hat{\varphi}(y)$ since $\hat{\varphi}$ preserves the distances between blocks. Finally we must check that $\hat{\varphi}$ is a digraph homomorphism. If $x,y \in V(f)$ for some $f \in VD$ then since $\hat{\varphi} \upharpoonright_{V(f)}$ is an isomorphism between two copies of $\Delta$ it follows that $(x,y) \in ED$ if and only if $(\hat{\varphi}(x), \hat{\varphi}(y)) \in ED$. Otherwise, $x$ and $y$ do not belong to a common block, which implies that they are not adjacent in $D$, and their images do not belong to a common block, and so are not adjacent in $D$ either, as required. It follows that $\hat{\varphi}: VD \rightarrow VD$ is a digraph isomorphism.
\end{proof}

This completes the proof of Theorem~\ref{bigtheorem}. 

\section{More constructions of connected-homogeneous digraphs}

In this section we present some examples showing that the $DL(\Delta)$ construction on its own is not enough to give a classification of infinitely ended locally-finite triangle-free $C$-homogeneous digraphs. 
In the process we obtain a new family of highly arc-transitive digraphs without property $Z$. 

We begin with an informal description of a family of digraphs which will be denoted $M(n,k)$ where $n,k \in \mathbb{N}$ with $n \geq 3$ and $k \geq 2$. Then we go on to give a formal definition of $M(n,k)$ below. The digraph $M(n,k)$ is $C$-homogeneous, and its reachability graph is isomorphic to the complement of a perfect matching $CP_n$. Before we proceed we need the following definition.  

\begin{defn}[Reachability intersection digraph]
Let $D$ be an arc-transitive digraph whose reachability graph $\Delta(D)$ is bipartite. The \emph{reachability intersection digraph} $\mathcal{R}(D)$ of $D$ has vertex set the $\mathcal{A}$-classes of $D$, and an arc $\mathcal{C}_1 \rightarrow \mathcal{C}_2$ if and only if there is a $2$-arc $x \rightarrow y \rightarrow z$ in $D$ with $(x,y) \in \mathcal{C}_1$ and $(y,z) \in \mathcal{C}_2$. Note that $\mathcal{R}(D)$ can have pairs of vertices with arcs between them going in both directions. We view such pairs as being joined by an undirected edge. 
\end{defn}

Our construction is based on that given by Malni\v{c} et al. in \cite[Section~2]{Malnic1}. 
The notion of reachability intersection digraph defined above generalises the definition of 
intersection graph given in \cite{Malnic1}. 
For each odd integer $n \geq 3$  the construction in \cite{Malnic1}
gives a highly arc-transitive digraph, denoted $Y_n$, without property $Z$. 
The digraph $Y_n$ has 
in- and out-degree equal to $2$, and the reachability digraph $\Delta(Y_n)$ is an alternating cycle of length $2n$. 

Following the description given in \cite{Malnic1}, to 
give an intuitive recursive definition of the digraph $Y_n$, start with an
alternating cycle $\Delta$ of length $2n$. At each pair of antipodal vertices $u$ and $v$ of
$\Delta$, glue an alternating cycle $\Delta_{u,v}$ of the same length to $\Delta$, in such a way
that $u$ and $v$ are antipodal on $\Delta_{u,v}$ as well, and that their in- and out-degrees
equal $2$. In this way, every vertex on $\Delta$ attains valency $4$ (in-degree $2$ and outdegree $2$). The process is repeated in each of the new alternating cycles and at each pair
of antipodal vertices of valency $2$. 
See \cite[Figure~1]{Malnic1} for an illustration of the digraph $Y_5$. 

The reachability digraphs of $Y_n$ are precisely the alternating cycles of length $2n$ from which it is built. 
It can be seen that any two adjacent alternating cycles in the above
construction intersect in precisely two antipodal vertices. That is, given any two reachability digraphs $\Delta_1$ and $\Delta_2$ in $Y_n$, if $\Delta_1$ and $\Delta_2$ are adjacent in the reachability intersection digraph then $|\Delta_1 \cap \Delta_2| = 2$. That $Y_n$ does not have property $Z$ can be seen by taking one half of each of these two alternating cycles $\Delta_1$ and $\Delta_2$ to obtain an unbalanced cycle. 

Now it may be easily verified that 
for $n \geq 4$ the digraph $Y_n$ is not $C$-homogeneous. To see this, fix a $2$-arc $(a,x,b)$ in $Y_n$ (where $n \geq 4$), let $\Delta_1$ be the reachability graph of the arc $a \rightarrow x$ (which by definition is isomorphic to $C_{2n}$) and let $\Delta_2$ be the reachability graph of the arc $x \rightarrow b$. Then $|\Delta_1 \cap \Delta_2| = 2$, say $\Delta_1 \cap \Delta_2 = \{ x,y \}$. 
(To visualise this, the reader is advised to look at \cite[Figure~1]{Malnic1} where $\Delta_1$ is the alternating cycle that goes around the perimeter of the circle and $\Delta_2$ is the `thin' alternating cycle attached to the perimeter at its top and bottom vertices.)
If $Y_n$ were $C$-homogeneous then, since $n \geq 4$, there would be an automorphism fixing each of the vertices in $\Delta_2$, fixing both the in-neighbours of $x$ (which belong to $\Delta_1$), but interchanging the two in-neighbours of $y$. This is clearly impossible since $\Delta_1$ is a cycle.

On the other hand, as we shall see below, the digraph $Y_3$ is $C$-homogeneous. Now we change our viewpoint slightly, by viewing the $6$-cycle $C_6$ as the complement of a perfect matching $CP_3$, this is possible since they happen to be isomorphic. In this way, the digraph $Y_3$ may be considered as being built up from copies of $CP_3$, glued together as follows. Begin with a copy $X \cup Y$ of $CP_3$ where $X$ and $Y$ are the parts of the bipartition, and all arcs are directed from $X$ to $Y$. Take a matched pair $(x,y) \in X \times Y$ with $x$ unrelated to $y$. Let $X' \cup Y'$ be another copy of $CP_3$, with arcs directed from $X'$ to $Y'$, and let $x', y'$ be an unrelated pair with $x' \in X'$ and $y' \in Y$. Now glue together $X \cup Y$ and $X' \cup Y'$ by identifying $x$ with $y'$, and $y$ with $x'$. So the two copies $X \cup Y$ and $X' \cup Y'$ of $CP_3$, intersect in a set of size $2$, and the arcs in the digraph $X' \cup Y'$ are oriented in the opposite direction to the arcs in $X \cup Y$. Repeating this process for every pair of matched vertices we obtain the tree-like structure illustrated in Figure~\ref{Y6}. This digraph is isomorphic to the digraph $Y_3$ from  \cite[Section~2]{Malnic1}. 
Drawn in this way the underlying reachability intersection digraph is made clearly visible, and is isomorphic to a trivalent tree. There is an obvious natural generalisation of this digraph where we glue together copies of $CP_n$ ($n \geq 3$) in the same tree-like way. We shall denote this digraph by $M(n,2)$ for $n \geq 3$. The underlying reachability intersection digraph of $M(n,2)$ is a tree of valency $n$. The construction may be generalised further still in such a way that the reachability intersection digraph is isomorphic to the Cayley graph of the free product of a finite number of copies of some fixed finite cyclic group, with respect to the natural generating set (see below). For example, we use $M(4,3)$ to denote the digraph illustrated in Figure~\ref{M43}, which is built up from copies of $CP_4$ and has reachability intersection digraph isomorphic to a tree of directed triangles, also illustrated in Figure~\ref{M43}.

\begin{figure}
\begin{center}
\resizebox{.4\textwidth}{!}{\includegraphics[angle=270]{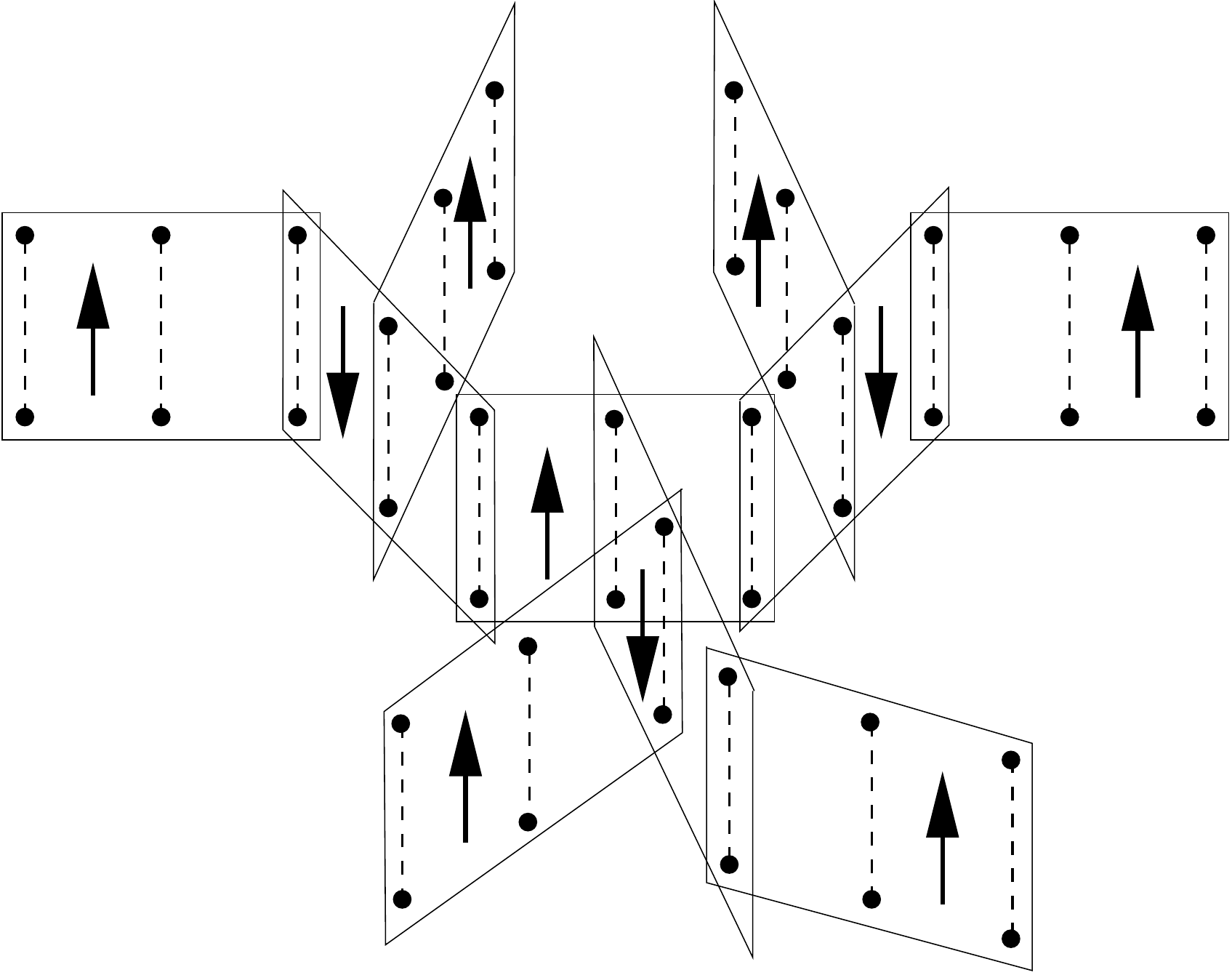}}
\end{center}

\caption{A partial view of the digraph $Y_3 \cong M(3,2)$ defined in \cite{Malnic1}. The arrows indicate the orientation of the arcs, while the dotted edges represent non-adjacent matched pairs of vertices, in the copies of $CP_3$.} \label{Y6}
\end{figure}

\begin{figure} \begin{center}
\begin{tabular*}{0.712\textwidth}{ | c | c | }
 \hline & \\
   
\includegraphics[scale=.35, angle=270]{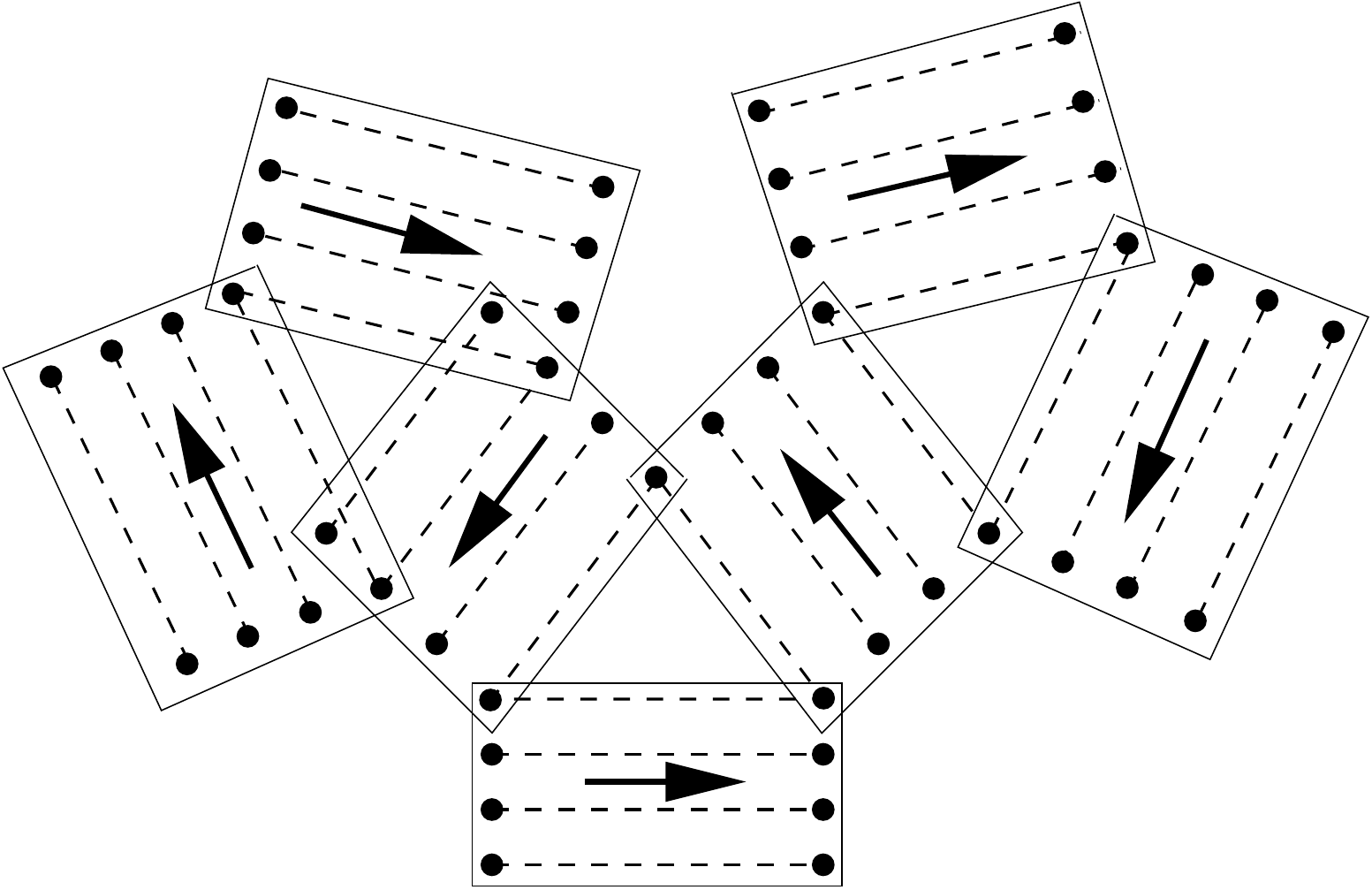} &  \includegraphics[scale=.35, angle=270]{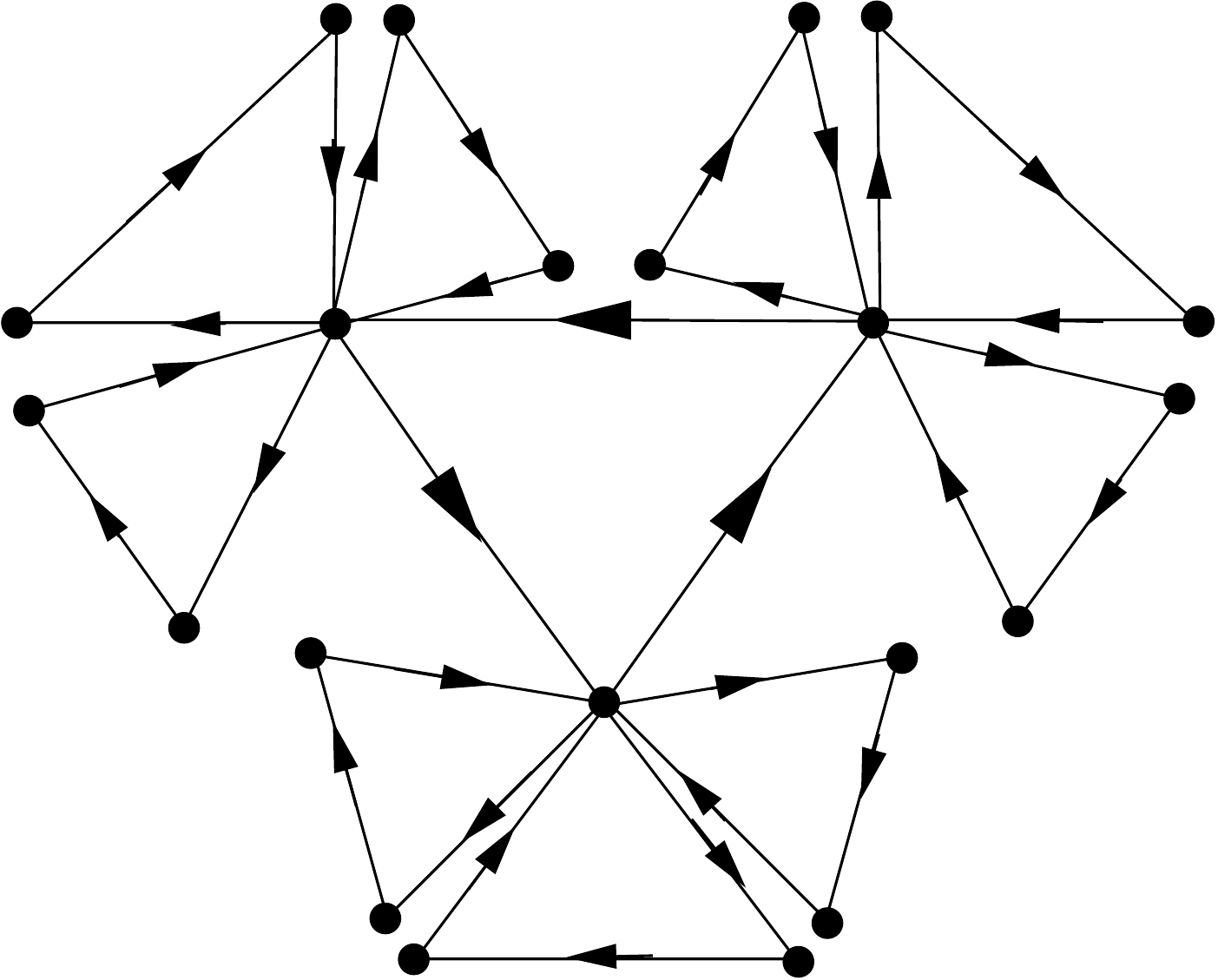}  \\
& \\  
\hline 
   
    $M(4,3)$   
    &  
   
    $\mathcal{R}(M(4,3))$ 
     \\ 

  \hline

\end{tabular*}

\end{center}
\caption[]{A partial view of the digraph $M(4,3)$ and its underlying reachability intersection digraph $\mathcal{R}(M(4,3))$. \label{M43}}
\end{figure}

In general $M(n,k)$ will be constructed in an analogous way to $M(4,3)$ but from copies of $CP_n$ and in such a way that the reachability intersection digraph is isomorphic to the digraph constructed in a tree-like way from gluing together directed $k$-cycles, where every vertex is the meeting point of precisely $n$ cycles. 

\subsection*{\boldmath Formal construction for the digraph $M(n,k)$}

The following definition is based on the construction given in \cite{Malnic1}, which it generalises. 
To the greatest extent possible we follow their notation and terminology. 
Let $n,k \in \mathbb{N}$ with $n \geq 3$, $k \geq 2$. Let $T_{n,k}$ be the directed Cayley graph of the free product of $n$ copies of the cyclic group $\mathbb{Z}_k$
\[
F_{n,k} = \langle a_1 \rangle * \langle a_2 \rangle * \cdots * \langle a_{n} \rangle, \quad a_i^k = 1, \quad i = 1,\ldots, n
\]
with respect to the generating set $A^{(n)} = \{ a_1, \ldots, a_{n}  \}$. 
Observe that every vertex of $T_{n,k}$ is the meeting point of exactly $n$ copies of the directed cycle $D_k$. 
Next we enlarge $T_{n,k}$ to a digraph $T_{n,k}^*$ obtained by replacing each vertex $v$ by a copy of the complete graph $K_n$, where each edge in $K_n$ is represented by arcs going in both directions. This is done in such a way that the $n$ copies of $D_k$ that used to all meet at the vertex $v$, are in  $T_{n,k}^*$ separated out with each one attached to exactly one of the $n$ vertices of $K_n$. 
The digraph $T_{n,k}^*$ has the property that each vertex $v$ has exactly one neighbour $\rightarrow$-related to it, one neighbour $\leftarrow$-related to it, and the remaining $n-1$ neighbours are both $\rightarrow$- and $\leftarrow$-related to $v$ (this is depicted as an undirected edge in Figure~\ref{T34} where $T_{4,3}^*$ is illustrated).
\begin{defn}
For two digraphs $\Gamma_1$ and $\Gamma_2$ we define their tensor product $\Gamma = \Gamma_1 \otimes \Gamma_2$ to be the digraph $\Gamma$ with vertex set $V\Gamma_1 \times V\Gamma_2$ and $((a,b),(c,d)) \in E\Gamma$  if and only if $(a,c) \in E\Gamma_1$ and $(b,d) \in E\Gamma_2$.
\end{defn}

(Note that the tensor product is also sometimes referred to in the literature as the direct product, categorical product or relational product.)

Next we construct the tensor product $T_{n,k}^* \otimes \vec{K_2}$ of $T_{n,k}^*$ by the directed graph $\vec{K_2}$ with two vertices $V\vec{K_2} = \{ a,b \}$ and one arc $a \rightarrow b$. In other words, $T_{n,k}^* \otimes \vec{K_2}$ is the canonical double cover of $T_{n,k}^*$ where all arcs are oriented from level $-$ to level $+$. Note that the copies of $K_n$ in $T_{n,k}^*$ have been transformed into copies of $CP_n$ in the digraph $T_{n,k}^* \otimes \vec{K_2}$. 

We say that an arc $(x,a) \rightarrow (y,b)$ of $T_{n,k}^* \otimes \vec{K_2}$ \emph{arises from a directed cycle of $T_{n,k}^*$} if the arc $x \rightarrow y$ belongs to one of the the directed cycles of $T_{n,k}^*$ (and not one of the copies of the complete graph $K_n$). 
In other words $x \rightarrow y$ in $T_{n,k}^*$ but 
that there is not an arc from $y$ to $x$.
Now consider just the arcs in $T_{n,k}^* \otimes \vec{K_2}$ that  arise from directed cycles of $T_{n,k}^*$. Observe that in $T_{n,k}^*$, not counting the undirected edges in the copies of $K_n$, every vertex has exactly one arc entering it and exactly one arc leaving it, where these two arcs belong to one of the copies of a directed cycle $D_k$ in $T_{n,k}^*$. It follows that the arcs in $T_{n,k}^* \otimes \vec{K_2}$ that  arise from directed cycles of $T_{n,k}^*$ are pairwise disjoint, and they define a bijection (i.e. a matching) between the two levels $VT_{n,k}^* \times \{ a \}$ and $VT_{n,k}^* \times \{ b \}$, where $V\vec{K_2} = \{ a,b \}$. 
Then, the digraph $M(n,k)$ is obtained by contracting all arcs in $T_{n,k}^* \otimes \vec{K_2}$ arising from directed cycles of $T_{n,k}^*$ (i.e. by identifying pairs of vertices related by the matching between the two levels $VT_{n,k}^* \times \{ a \}$ and $VT_{n,k}^* \times \{ b \}$ described above). The digraph $M(n,k)$ has in-degree and out-degree equal to $n-1$, its reachability graph is isomorphic to $CP_n$, and the reachability intersection digraph is isomorphic to $T_{n,k}$.

\begin{figure}
\begin{center}
\resizebox{.4\textwidth}{!}{\includegraphics[angle=270]{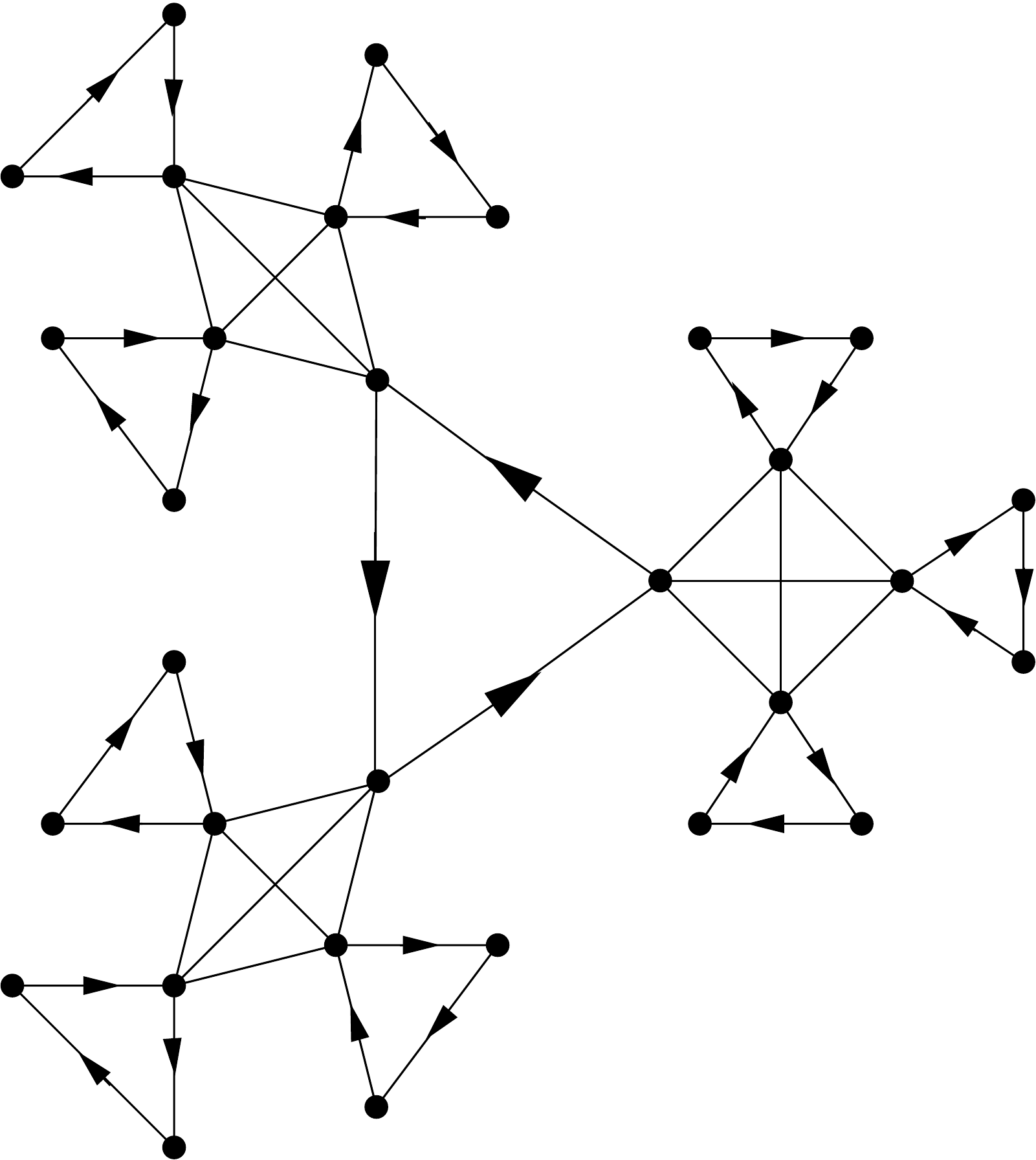}}
\end{center}
\caption{A partial view of the digraph $T_{4,3}^*$.} \label{T34}
\end{figure}

\begin{prop}
The digraph $M(n,k)$ is $C$-homogeneous for $n \geq 3, k \geq 2$.
\end{prop}
\begin{proof}
Let $M = M(n,k)$, $T = T_{n,k}$ and let $\pi : EM \rightarrow VT$ map each arc $e$ of $M$ to the vertex of $T$ corresponding to the $\mathcal{A}$-class $\mathcal{A}(e)$. For any connected induced substructure $M_1$ of $M$, with at least one arc, we use $\pi(M_1)$ to denote the image of the set of arcs of $M_1$ under $\pi$. 

Note the $T$ itself is not (in general) $C$-homogeneous. Indeed, when $k \geq 4$, there are two types of induced directed $2$-arc in $T$: those that embed in a copy of the directed cycle $D_k$ in $T$, and those that do not. However, this is essentially the only obstruction to the $C$-homogeneity of $T$. More precisely, let $\sim$ be the equivalence relation on $ET$ defined by $e \sim f$ if and only if there is an induced copy of $D_k$ in $T$ containing both $e$ and $f$. So the $\sim$-classes of $ET$ are just the copies of $D_k$ in $T$. Then we say that an isomorphism $\phi: T_1 \rightarrow T_2$ between finite connected induced substructures $T_1$, $T_2$, of $T$ \emph{respects $\sim$} if for all $(a,b), (c,d) \in ET_1$ we have
\[
(a,b) \sim (c,d) \Leftrightarrow (\phi(a),\phi(b)) \sim (\phi(c),\phi(d)).
\]
It is then not hard to see that an isomorphism $\phi:T_1 \rightarrow T_2$ between finite connected induced subdigraphs $T_1$, $T_2$ of $T$ extends to an automorphism of $T$ provided $\phi$ respects $\sim$ (in the case $k \geq 4$ this is an necessary and sufficient condition for $\phi$ to extend). 

Turning our attention to $M$, every isomorphism $\varphi: M_1 \rightarrow M_2$ between connected induced subdigraphs $M_1$, $M_2$ of $M$ (where $M_1$ has at least one arc) induces, in the obvious way, an isomorphism $\hat{\varphi}: \lb \pi(M_1) \rb \rightarrow \lb \pi(M_2) \rb$ where $\lb \pi(M_1) \rb$ and $\lb \pi(M_2) \rb$ are both finite connected induced subdigraphs of $T$. 
In particular, any automorphism $\varphi \in \Aut(M)$ induces an automorphism $\hat{\varphi} \in \Aut(T)$.

\bigskip

\noindent \textbf{Claim (a).} \textit{The mapping $\hat{\;}: \Aut(M) \rightarrow \Aut(T)$, $\varphi \mapsto \hat{\varphi}$ is an isomorphism of groups.}

\bigskip

\begin{proof}[Proof of Claim (a)] This map is clearly a homomorphism. 

To see that $\hat{\;}$ is injective it suffices to observe that the only automorphism of $M$ that fixes setwise the $\mathcal{A}$-classes of $M$ is the identity map. The easiest way to see this is to consider the extra structure on $M$ coming from the fact that the reachability digraphs are isomorphic to $CP_n$. We do this by defining a binary relation $\Rightarrow$ on $VM$ by:
\[ \mbox{
$x \Rightarrow y$ if and only if there exist $z,t \in VM$ satisfying $x \rightarrow z \leftarrow t \rightarrow y$ and $x \not\rightarrow y$. 
} \]
In other words, $x \Rightarrow y$ if and only if $x$ and $y$ are both vertices of some $\mathcal{A}$-class $\mathcal{A}(e)$, $x$ and $y$ both belong to different parts of the bipartition of $\mathcal{A}(e)$, and $x$ and $y$ are unrelated by $\rightarrow$ (i.e. $x$ and $y$ are ``matched'' in the bipartite complement of perfect matching $CP_n$). Moreover, the direction of $\Rightarrow$ corresponds to the orientation of the arcs in $\mathcal{A}(e)$. In the diagram in Figure~\ref{M43} the dotted edges correspond to $\Rightarrow$-related pairs. 
We shall call a pair $(x,y)$ with $x \Rightarrow y$ and $\Rightarrow$-arc. 

From the construction of $M$ we see that $(VM, \Rightarrow)$ has the structure of countably many disjoint $\Rightarrow$-directed $k$-cycles. Moreover, if $x \Rightarrow y \Rightarrow z$, $x' \Rightarrow y' \Rightarrow z'$ and $\{ x,y,x',y' \}$ is a subset of the vertices of some $\mathcal{A}$-class, then it follows from the structure of $M$ that $\{y,z\}$ and $\{y',z'\}$ are subsets of different $\mathcal{A}$-classes. 

Now let $\alpha \in \Aut(M)$ be an automorphism that fixes setwise the $\mathcal{A}$-classes of $M$. Then, on each $\mathcal{A}$-class, $\alpha$ induces an automorphism, which is determined by some permutation of the $\Rightarrow$-arcs of the $\mathcal{A}$-class. However, if for some $\mathcal{A}$-class a non-identity permutation is induced by $\alpha$ on its $\Rightarrow$-arcs then, by the observation in the previous paragraph, the automorphism $\alpha$ does not fix the $\mathcal{A}$-classes of $M$ setwise, contradicting the choice of $\alpha$. From this it follows that the homomorphism $\hat{\;}: \Aut(M) \rightarrow \Aut(T)$ is injective, since $\hat{\alpha} = \hat{\beta}$ implies that $\alpha \beta^{-1} \in \Aut(M)$ is an automorphism fixing setwise the $\mathcal{A}$-classes of $M$, giving $\alpha \beta^{-1} = 1$.    

To see that $\hat{\;}$ is surjective, fix some $v \in VT$. Its preimage $\Delta = \phi^{-1}(v)$ in $M$ is a copy of $CP_n$ in $M$. Now consider $\Aut(T)_v$. These automorphisms are easily described. An automorphism $\alpha \in \Aut(T)_v$ is defined by choosing a permutation of the $n$ copies of $D_k$ that are adjacent to $v$, and then for each of the vertices on each such $D_k$, choosing a permutation of the $n-1$ other copies of $D_k$ attached to that vertex, and so on working our way out from the original vertex (this is analogous to considering the stabiliser of a vertex of an infinite regular tree). Correspondingly, we may consider $\Aut(M)_\Delta \leq \Aut(M)$: the automorphisms that fix $\Delta$ setwise. These are given by first choosing some permutation of the $n$ $\Rightarrow$-related pairs in $\Delta$, and then for every $\Rightarrow$-directed $k$-cycle $\mathcal{D}$ attached to $\Delta$, and for every $\mathcal{A}$-class $\Omega$ of each $\Rightarrow$-arc of $\mathcal{D}$, we choose freely some permutation of the $n-1$ $\Rightarrow$-arcs in $\Omega$ that are not in $\mathcal{D}$, and so on working our way out from $\Delta$.  
In this way we see that  $\hat{\;}: \Aut(M) \rightarrow \Aut(T)$ maps $\Aut(M)_\Delta$ onto $\Aut(T)_v$. Finally, to generate all of $\Aut(T)$, we observe that $\Aut(M)$ acts transitively on the $\mathcal{A}$-classes of $M$ and hence acts vertex transitively on $T$. We conclude that $\hat{\;}$ is surjective and thus  $\hat{\;}: \Aut(M) \rightarrow \Aut(T)$ is an isomorphism, completing the proof of the claim. \end{proof}

Our aim is to show that $M$ is $C$-homogeneous. One key idea for the proof is the following:

\bigskip

\noindent \textbf{Claim (b).} \textit{
For every isomorphism $\varphi : M_1 \rightarrow M_2$ between finite connected induced subdigraphs $M_1$, $M_2$ of $M$ the induced isomorphism $\hat{\varphi}: \lb \pi(M_1) \rb \rightarrow \lb \pi(M_2) \rb$ is an isomorphism between finite connected induced subdigraphs of $T$ that respects $\sim$, and therefore $\hat{\varphi}$ extends to an automorphism of $T$.
}

\bigskip

\begin{proof}[Proof of Claim (b)]
Let $M_1$ be a finite connected induced subdigraph of $M$ and let $(e,f)$ and $(e',f')$ be $2$-arcs in $M_1$, where $e,f,e',f' \in EM_1$. Then $(\pi(e),\pi(f)) \sim (\pi(e'), \pi(f'))$ in $T$ if and only if there is an \emph{induced} subdigraph of $M_1$ of one of the following two forms:
\begin{equation}
\label{eqn_path1}
\circ \xrightarrow{e}  \circ \xrightarrow{f}  \circ \leftarrow 
\circ \rightarrow      \circ \rightarrow      \circ \leftarrow 
\circ \cdots
\circ \rightarrow      \circ \rightarrow      \circ \leftarrow 
\circ \xrightarrow{e'} \circ \xrightarrow{f'} \circ
\end{equation}
or 
\begin{equation}
\label{eqn_path2}
\circ \xrightarrow{e'}  \circ \xrightarrow{f'}  \circ \leftarrow 
\circ \rightarrow      \circ \rightarrow      \circ \leftarrow 
\circ \cdots
\circ \rightarrow      \circ \rightarrow      \circ \leftarrow 
\circ \xrightarrow{e} \circ \xrightarrow{f} \circ.
\end{equation}   
This follows from inspection of the structure of $M$, and in particular relies on the fact that the reachability digraphs are isomorphic to $CP_n$ and $M_1$ is connected.   
The statement in the claim is then an immediate corollary. 
\end{proof}

Now we shall use the above observations to show that $M$ is $C$-homogeneous. Let $\varphi: M_1 \rightarrow M_2$ be an isomorphism between finite connected induced subdigraphs of $M$. 
Then, by Claim (b), $\hat{\varphi}: \lb \pi(M_1) \rb \rightarrow \lb \pi(M_2) \rb$  extends to some automorphism $\theta$ of $T$ (note that in general there will be more than one possible choice for $\theta$). Let $\sigma \in \Aut(M)$ satisfy $\hat{\sigma} = \theta$. This is possible by Claim (a). 
Then we see that $\varphi$ extends to an automorphism of $M$ if and only if  
$\varphi' = \varphi \circ (\sigma^{-1}\rrestriction_{M_2})$ extends to an automorphism of $M$. Therefore we may suppose without loss of generality (by replacing $\varphi$ by $\varphi'$) that for every arc $e$ in $EM_1$ we have:
\begin{equation}\label{eqn_double_hash}
\pi( \varphi(e) ) = \pi( e ).
\end{equation}  
Next observe that for every arc $f$ in $T$ there is a unique copy of the directed cycle $D_k$ in $T$ to which $f$ belongs. Therefore for every finite connected induced subdigraph $Z$ of $T$ there is a uniquely determined extension $\overline{Z}$ of $Z$ obtained by adding in all arcs $\sim$-related to arcs of $Z$ (i.e. we close under $\sim$). Clearly $\overline{Z}$ is then a connected union of finitely many copies of $D_k$ in $T$. 

From assumption \eqref{eqn_double_hash} and Claim (b) it follows that $\overline{\pi(M_1)} = \overline{\pi(M_2)}$. Now define $X = \pi^{-1}(\overline{\pi(M_1)}) = \pi^{-1}(\overline{\pi(M_2)})$, which is a preimage of a connected union of finitely many directed cycles of $T$. So, $M_1$ and $M_2$ are isomorphic finite connected substructures of $X$ and by \eqref{eqn_double_hash} the isomorphism $\varphi: M_1 \rightarrow M_2$ fixes setwise the $\mathcal{A}$-classes of $X$. We shall now prove that $\varphi$ extends to an automorphism of $X$. 

There are two types of vertices in $M_1$. We call a vertex $v$ of $M_1$ a \emph{corner vertex} if and only if $v$ belongs to a $\Rightarrow$-directed $k$-cycle in $X =  \pi^{-1}(\overline{\pi(M_1)})$. Similarly we define corner vertices of $M_2$. 

\bigskip

\noindent \textbf{Claim (c).} \textit{
For all $v \in VM_1$, $v$ is a corner vertex of $M_1$ if and only if $\varphi(v)$ is a corner vertex of $M_2$, in which case $\varphi(v) = v$. 
}

\bigskip

\begin{proof}[Proof of Claim (c)]
A vertex $w$ of $M_1$ is a corner vertex if and only if either $w$ is the middle vertex of a directed $2$-arc of $M_1$, or  there is an induced subdigraph of $M_1$ of one of the following two forms:
\begin{equation}
\label{eqn_path1'}
\circ \xrightarrow{}  \circ \xrightarrow{}  \circ \leftarrow 
\circ \rightarrow  w    
\end{equation}
or 
\begin{equation}
\label{eqn_path2'}
w \xrightarrow{}  \circ \leftarrow 
\circ \rightarrow      \circ \rightarrow      \circ. 
\end{equation}   
To see this, let $w$ be an arbitrary corner vertex of $M_1$. 
If $w$ is the middle vertex of a directed $2$-arc then clearly $w$ is a corner vertex, so suppose otherwise. 
Take a $\Rightarrow$-directed $k$-cycle in $X$ containing $w$. Then since this $\Rightarrow$-directed $k$-cycle is in $X$, some directed $2$-arc, with middle vertex $v$ say, of $M_1$ must embed in the union of the $\mathcal{A}$-classes traversed by the $\Rightarrow$-directed $k$-cycle. But $M_1$ is connected, and taking a path in $M_1$ from the directed $2$-arc to $w$, removing unnecessary edges and applying the fact that $\Delta \cong CP_n$, we conclude that there is an induced subdigraph of $M_1$ of one of the following two forms:     
\begin{equation}
\label{eqn_path1''}
\circ \xrightarrow{}  v \xrightarrow{}  \circ \leftarrow 
\circ \rightarrow     
\circ \rightarrow      \circ \leftarrow 
\circ \cdots
\circ \rightarrow      \circ \rightarrow      \circ \leftarrow 
\circ \xrightarrow{} w
\end{equation}
or 
\begin{equation}
\label{eqn_path2''}
w \xrightarrow{}  \circ \leftarrow 
\circ \rightarrow      \circ \rightarrow      \circ
\leftarrow 
\circ \cdots
\circ \rightarrow      \circ \rightarrow      \circ \leftarrow 
\circ \xrightarrow{} v  \rightarrow \circ,
\end{equation}   
then cutting down to $4$-element substructures proves the result. Similarly the corner vertices of $M_2$ may be described. From this description it is clear that, since $\varphi$ is an isomorphism, $v$ is a corner vertex of $M_1$ if and only if $\varphi(v)$ is a corner vertex of $M_2$. 

Now let $v \in VM_1$ be a corner vertex of $M_1$. Recall that $\varphi: M_1 \rightarrow M_2$ is assumed to fix setwise the $\mathcal{A}$-classes of $X$. 
It follows that:

(i) if $v$ is the middle vertex of a $2$-arc then we must have $\varphi(v) = v$, since in $M$ every pair of $2$-arcs between a given ordered pair of $\mathcal{A}$-classes in $M$ have the same middle vertex; 

(ii) since every vertex $w$ in $M$ has a unique $\Rightarrow$-outneighbour and unique $\Rightarrow$-inneighbour it follows that once we know $\varphi(v) = v$ we must have $\varphi(w) = w$, where $w$ and $v$ are as in \eqref{eqn_path1''} or \eqref{eqn_path2''}. 

We conclude that $\varphi(w) = w$ for every corner vertex $w$ of $M_1$. \end{proof}

Then the remaining non-corner vertices may be permuted at will within their respective $\mathcal{A}$-classes, and in this way $\varphi$ is determined. But then $\varphi$ is clearly the restriction of some automorphism of $X$ fixing setwise the $\mathcal{A}$-classes of $X$. This completes the proof that $\varphi$ extends to an automorphism of $X$. 

Finally, using the fact that $X$ is the preimage under $\pi$ of a connected union of finitely many copies of $D_k$ in $T$, it is not difficult to see that every automorphism of $X$ extends to an automorphism of $M$. 

We conclude that $M$ is $C$-homogeneous. \end{proof}

The above construction provides 
a new family of locally-finite highly arc-transitive digraphs with no homomorphism onto $Z$. 

\begin{corol}
The digraph $M(n,k)$ $(n \geq 3, k \geq 2)$ is highly arc-transitive with in- and out-degree equal to $n-1$, and does not have property $Z$.
\end{corol}

We finish this section by describing a different family of $C$-homogeneous digraphs, that are constructed using the line digraph operation. 

\begin{defn}[Line digraph] The line digraph $L(D)$ of a digraph $D$ is defined to be the digraph with vertex set $ED$ and (directed) edges of the form $(e,e')$, where the arcs $e$, $e'$ give rise to a $2$-arc in $D$. 
\end{defn}  

It was proved in \cite[Lemma~4.1(a)]{Cameron2} that high-arc-transitivity is a property that is preserved when taking the line digraph of a digraph. In general, connected-homogeneity is not preserved under taking line digraphs. The digraph $J(2)$, which is illustrated in Figure~\ref{J(2)} and will be formally defined below in Section~\ref{sec_2ended}, is $C$-homogeneous.
A straightforward check shows that  the  line graph $L(J(2))$ of the digraph $J(2)$, illustrated in Figure~\ref{J(2)}, is not $C$-homogeneous.  

Indeed, the digraph $L(J(2))$ is easily seen to be a connected locally-finite digraph with two ends. So if $L(J(2))$ were $C$-homogeneous then, by Theorem~\ref{2_ended_classification} below, this would imply that $L(J(2)) \cong J(r)$ for some $r$ (where $J(r)$ is defined below). This is clearly not the case, since the reachability digraph of $L(J(2))$ is a bipartite digraph with bipartitions each of size $4$, while the in- and out-degrees of the vertices of $L(J(2))$ are only equal to $2$. Therefore $L(J(2))$ is not $C$-homogeneous.    

However, for $m \geq 2$ it may be verified that the line graph $L(DL(C_{2m}))$ is $C$-homogeneous, where $C_{2m}$ denotes the alternating cycle with $2m$ vertices. Note that $L(DL(C_{2m}))$ has reachability digraph isomorphic to $K_{2,2}$. The digraphs $L(DL(C_{2m}))$ for $m \geq 2$ do have property $Z$.

\section{$2$-ended connected-homogeneous digraphs}
\label{sec_2ended}

We begin by describing a family of examples. 

\begin{defn}
For $r \in \mathbb{N}$ let $J(r)$ denote the digraph with vertex set $\mathbb{Z} \times X$, where $X = \{ 1, \ldots, r \}$, and with arcs $(i,x) \rightarrow (i+1,y)$ where $i \in \mathbb{Z}$ and $x,y \in X$. 
\end{defn}

\begin{figure}
\[
\xymatrix{
\dedge{r}  &  \node \rulab{} \arcc{r} \arcc{dr} & \node \rulab{} \arcc{r} \arcc{dr} & \node \rulab{} \arcc{r} \arcc{dr} & \node  \rulab{} \dedge{r} & \\
\dedge{r} &  \node \rdlab{} \arcc{r} \arcc{ur} &  \node \rdlab{} \arcc{r} \arcc{ur} &  \node \rdlab{} \arcc{r} \arcc{ur} &  \node  \rdlab{} \dedge{r} &
}
\]
\caption{The digraph $J(2)$.} \label{J(2)}
\end{figure}
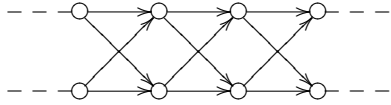

The digraph $J(2)$ is illustrated in Figure~\ref{J(2)}.
It is easy to see that for every $r \in \mathbb{N}$ the digraph $J(r)$
has two ends and is $C$-homogeneous. Indeed, the group generated by
the translates, mapping $(i,x) \mapsto (i+k,x)$ for $k \in
\mathbb{Z}$, and the transpositions, interchanging  $(i,x)$ with
$(i,y)$ and fixing every other vertex, acts in a $C$-homogeneous way
on the digraph $J(r)$.    

The main result of this section is the following. 

\begin{thm}\label{2_ended_classification}
Let $D$ be a connected locally-finite digraph with exactly two
ends. Then $D$ is $C$-homogeneous if and only if it is isomorphic to
$J(r)$ for some $r \in \mathbb{N}$.  
\end{thm}

Before proving Theorem~\ref{2_ended_classification} we first 
need three lemmas. Let $Z$ denote the digraph with the set of integers as vertex set
and arcs $i\rightarrow i+1$.

\begin{lem} 
Let $D$ be a connected locally-finite $C$-homogeneous digraph
with exactly two ends. Then $D$ is triangle-free and there is a
surjective digraph homomorphism $\theta: D \rightarrow Z$. 
Moreover, the fibres $\theta^{-1}(i)$ of the homomorphism $\theta$ are finite for all $i \in \mathbb{Z}$.  
\end{lem}
\begin{proof}
Let $e_0 \subseteq VD$ be a $D$-cut, $E = Ge_0 \cup G
{e_0}^*$ the associated tree set where $G = \Aut(D)$, and  $T = T(E)$  the
structure tree. Let $\phi : VD \rightarrow VT$ be the structure
mapping.  The structure tree $T$ has just two ends and has no leaves,
and is therefore a line.  There is a positive integer $N$ such that whenever $x$
and $y$ are related vertices in $D$ then $d_T(\phi(x), \phi(y))=N$.
If the vertices $x, y, z$ form a triangle in $D$ then $\phi(x),
\phi(y)$ and $\phi(z)$ are vertices in the line $T$ such that the distance
between any two of them is $N$.  This is clearly impossible so 
$D$ cannot contain any triangles.  By Theorem~\ref{bigtheorem}(i) the
digraph $D$ is highly arc-transitive.  

It is immediate from the definition of $E$ that the group $G=\aut(D)$ acts on $T$ and is transitive on the undirected 
edges in $T$, by which we mean for all $e,f \in E$ either $e$ and $f$ belong to the same $G$-orbit, or $e$ and $f^*$ belong to the same $G$-orbit. 
Let $v \in VT$. In the set of arcs of $T$ there are only two arcs, say
$f$ and  $f'$,  that have $v$ as a terminal vertex.  Either $f$ is fixed by
$G_v$ or the orbit of $f$ under $G_v$ has just two elements $\{f, f'\}$.  This
means that the finite set $\delta f\cup \delta f'$ 
(here both the cuts $f$ and $f'$ are being viewed as subsets of $VD$) 
of arcs in $D$ is
invariant under $G_v$ (where $\delta f$ denotes the co-boundary of $f$ as defined in Section~2).  Hence $G_v$ has a finite orbit on the arcs of
$D$ and thus only finite orbits on the vertex set of $D$.  The
fiber $\phi^{-1}(v)$ is one of the orbits of $G_v$ and thus is finite.
Because $G$ acts transitively on $\im{\phi}$ and the action of $G$ on
$D$ is covariant with $\phi$  we see that all the
fibers of $\phi$ have the same finite number of elements, say $K$.  

Fix a vertex $u$ in $D$.  Let $p_j$ denote the number of vertices in
$D^j(u)$ (for $j \in \mathbb{Z}$).  By Lemma~\ref{just2ends} the map $\phi$ is constant on
$D^j(u)$ and thus $p_j=|D^j(u)|\leq K$.  Hence the out-spread of $D$,
defined as $\limsup_{j\rightarrow \infty}(p_j)^{1/j}$, is equal to 1.
By \cite[Theorem 3.6]{Cameron2} the digraph $D$ has property $Z$, that
is to say there is a surjective digraph homomorphism $\theta:D\rightarrow Z$.

The last part essentially follows from Lemma~\ref{just2ends} and its proof. 
Indeed, we claim that for all $u,v \in VD$, if $\theta(u) = \theta(v)$ then $\phi(u) = \phi(v)$. Then, since the fibres of $\phi$ are finite it will follow that the fibres of $\theta$ are also finite. So, suppose that $u,v \in \theta^{-1}(i)$ for some $i \in \mathbb{Z}$. We claim that $\phi(u) = \phi(v)$. Since $D$ is connected we can choose an undirected path $\pi$ in $D$ from $u$ to $v$. Let $u = u_0, u_1, \ldots, u_k=v$ be the vertices of the path $\pi$. We call a vertex $u_j$ in the path $\pi$ a turning point if either $u_{j-1} \rightarrow u_j \leftarrow u_{j+1}$ or $u_{j-1} \leftarrow u_j \rightarrow u_{j+1}$. We shall prove by induction on the number of turning points in $\pi$ that $\phi(u) = \phi(v)$. If $\pi$ has just one turning point then $\phi(u) = \phi(v)$ by Lemma~\ref{just2ends}. For the induction step, if there is a vertex $u_l \in \theta^{-1}(i)$ with $0<l<k$ then applying induction to each of the paths $u = u_0, u_1, \ldots, u_l$ and $u_l, \ldots, u_k = v$ we deduce $\phi(u) = \phi(u_l) = \phi(v)$, as required. Otherwise, without loss of generality we may suppose that the path $\pi$ is contained in $\cup_{m \leq i} \theta^{-1}(m)$. In this case let $u_{j_2}$ be the second turning point of the path $\pi$ and choose a directed path $\pi'$ from $u_{j_2}$ to some vertex $w \in \theta^{-1}(i)$. Such a path $\pi'$  exists since $\theta$ is a homomorphism, and every vertex in $D$ has outdegree at least one. Then we may apply induction to the paths $u=u_0, \ldots, u_{j_2}, \pi'$ and $(\pi')^{-1}, u_{j_2}, \ldots, u_k = v$ to deduce $\phi(u) = \phi(w) = \phi(v)$, completing the induction step. 
\end{proof}

\begin{lem}Let $D$ be a connected locally-finite $C$-homogeneous digraph
with exactly two ends. 
Let $\theta$ be a surjective digraph homomorphism
  $D\rightarrow Z$.  Then for every integer $j$ 
the subdigraphs induced by $\cup_{i\leq j-1}\theta^{-1}(i)$ and 
$\cup_{i\geq j+1}\theta^{-1}(i)$ are both connected.
\end{lem}

\begin{proof}
The graph $D\setminus \theta^{-1}(j)$ is not connected, since every
path from a vertex in $\theta^{-1}(j-1)$ to a vertex in
$\theta^{-1}(j+1)$ must include a vertex in $\theta^{-1}(j)$.  The
digraph $D$ is assumed to have precisely two ends so if we remove a
finite set of vertices from $D$ then we get at most two infinite
components.   If $u$ is a vertex in $\cup_{i\leq j-1}\theta^{-1}(i)$
then the set $\anc(u)$ belongs to the same component of 
$D\setminus\theta^{-1}(j)$ as $u$ does and this set is infinite.
Similarly for a vertex $u$ in $\cup_{i\geq j+1}\theta^{-1}(i)$ the set
$\desc(u)$ belongs to the same component of
$D\setminus\theta^{-1}(j)$ as $u$.   Thus $D\setminus\theta^{-1}(j)$
has no finite components and there are at most two infinite components 
so the subdigraphs induced by  
$\cup_{i\leq j-1}\theta^{-1}(i)$ and 
$\cup_{i\geq j+1}\theta^{-1}(i)$ must be connected.
\end{proof}

\begin{lem}\label{completebipartite}
Let $D$ be a connected locally-finite $C$-homogeneous digraph
with exactly two ends. 
Let $\theta$ be a surjective digraph homomorphism
  $D\rightarrow Z$.  Then
there exist $k, l \in \mathbb{N}$ such that for all $i \in \mathbb{Z}$,
the subdigraph induced by $\theta^{-1}(i) \cup \theta^{-1}(i+1)$ is
isomorphic to the disjoint union of $k$ copies of the complete
bipartite graph $K_{l,l}$. In particular, $\Delta(D) \cong K_{l,l}$.     
\end{lem}
\begin{proof}   
 Define $\mathcal{B}$ as the graph 
$\lb \theta^{-1}(0) \cup \theta^{-1}(1) \rb$,
noting that $\mathcal{B}$ is finite since the fibres of $\theta$ are finite.
Note that each of the graphs 
$\mathcal{B}_k=\lb \theta^{-1}(k) \cup \theta^{-1}(k+1)
\rb$ is isomorphic to the graph $\mathcal{B}$. Of course all
the arrows of the bipartite graph $\mathcal{B}$ are oriented in
the same way from $\theta^{-1}(0)$ to $\theta^{-1}(1)$.  
Clearly $\mathcal{B}$ is a disjoint union of a finite number of copies
of the reachability bipartite graph $\Delta$. 

Next we shall prove that $\Delta = \Delta(D)$ is isomorphic to
$K_{l,l}$ for some $l \in \mathbb{N}$.  
Certainly $\Delta$ is finite since $\Delta \subseteq \mathcal{B}$
which is finite  and so $\Delta \not\cong T_{a,b}$. 
So by Theorem~\ref{bigtheorem}(iii) it suffices to prove that $\Delta$
is not the complement of a perfect matching (with at least $4$
vertices), and is not an even cycle $C_m$ with $m \geq 8$.  Once this
has been established it will follow that $\Delta \cong K_{l,m}$ for
some $l,m \in \mathbb{N}$. Then because the
fibers $ \theta^{-1}(i) $ (for $i \in \mathbb{Z}$) all have the same
size, we may conclude that $l=m$. 

First we suppose that  
$\Delta$ is the complement of a perfect matching. So for each vertex
$v$ in $\theta^{-1}(i)$ there is a unique vertex $u$ in $\theta^{-1}(i+1)$
that is in the same
component of $\mathcal{B}_i$ but is not related to $v$.  For each
such pair we put in a new arc $u\rightarrow v$ and we call these arcs
{\em red arcs} to distinguish them from the original arcs in $D$.  Each
vertex in $D$ has precisely one in-going red arc and one out-going red
arc.  The digraph consisting of the vertices in $D$ and the  red arcs 
is therefore a collection of disjoint directed lines such
that each vertex in $D$ belongs to precisely one of these lines and
each of the lines contains precisely one vertex from each fiber
$\theta^{-1}(i)$ .  The
automorphism group of $D$ also preserves the red arcs and permutes these
lines.  Let $u$ be a vertex in $\theta^{-1}(0)$ and let $a$ and $b$ be
vertices in $D^+(u)$.  Denote with $L_a$ and $L_b$ the red lines that
$a$ and $b$ belong to, respectively.  
Since $\cup_{j\leq -1}\theta^{-1}(j)$ is connected
it is possible to find in $\cup_{j\leq 0}\theta^{-1}(j)$ a path $P$ starting
in a vertex in $L_a$ and ending in $u$ with the property that $P$ only contains one 
vertex from $\theta^{-1}(0)$, and so 
$P \cap \theta^{-1}(0) = \{ u \}$. 
The subdigraphs of $D$ induced
by $P\cup\{a\}$ and $P\cup\{b\}$ are isomorphic.  By $C$-homogeneity
    there is an automorphism that 
fixes    
all the vertices in $P$ and takes
$a$ to $b$.  This automorphism will take $L_a$ to $L_b$, but
that can not happen because the automorphism also fixes a vertex in
$P$ that belongs to $L_a$ and must therefore fix $L_a$.  Hence
the assumption that the $\Delta$ is the complement of a perfect
matching is untenable.

Now suppose that $\Delta \cong C_m$ for some $m \geq 8$. In particular each vertex in $D$ has in-degree and 
out-degree equal to $2$. Let $a,d \in \phi^{-1}(0)$ be distinct vertices such that they have a common neighbour $b \in \phi^{-1}(1)$. Also let $c \in D^+(a) \setminus \{ b \}$ be the other neighbour of $a$ in $\phi^{-1}(1)$. Let $u,v \in \phi^{-1}(-1)$ be the vertices of $D^{-1}(d)$. Let $w \in D^{-1}(a)$ noting that $w$ could belong to the set $\{u,v\}$, and  let $Q$ be a finite subset of $\cup_{j \leq -1} \phi^{-1}(j)$ such that $\{u,v,w\} \subseteq Q$ and $\lb Q \rb$ is connected. Such a set $Q$ exists since the subdigraph induced by $ \cup_{j \leq -1} \phi^{-1}(j) $ is connected. By $C$-homogeneity there is an automorphism $\alpha$ that fixes each of the vertices in $Q \cup \{ a \}$ and interchanges $b$ and $c$. However, since $\alpha$ fixes $u$ and $v$ it must also fix $d$ since $d$ is their only common neighbour in $\phi^{-1}(0)$. But this is impossible since $b \sim d$ but $c \not\sim d$ since $\Delta$ has at least $8$ vertices. \end{proof}

\begin{proof}[Proof of Theorem~\ref{2_ended_classification}] 

We will use the same notation as above.  
By Lemma~\ref{completebipartite}  it suffices to show  
the digraph $\mathcal{B} = \langle \theta^{-1}(0) \cup
\theta^{-1}(1) \rangle$ is connected. Seeking a contradiction, suppose
that  $\mathcal{B}$ has
more than one connected component. As observed above
 the digraph induced by $\bigcup_{j \geq 0}
\theta^{-1}(j)$ is connected. From this, together with $C$-homogeneity,
it follows that there are
distinct connected components $A$ and $B$ of
$\mathcal{B}$ such that for some $a \in A \cap \theta^{-1}(1)$ and
$b \in B \cap \theta^{-1}(1)$ we have $D^+(a) \cap D^+(b) \neq
\varnothing$. Let $c \in D^+(a) \cap D^+(b)$.   So $c \in
\theta^{-1}(2)$. By Lemma~\ref{completebipartite}, $\langle A \rangle
\cong \langle B \rangle \cong K_{l,l}$ for some $l \in
\mathbb{N}$. Since the subdigraph induced by $\bigcup_{i \leq 0}
\theta^{-1}(i)$ is connected, there exists a finite set $F \subseteq
\bigcup_{i \leq 0} \theta^{-1} (i)$ such that $\theta^{-1}(0) \subseteq F$
and $\langle F \rangle$ is connected. Now for every $a' \in A \cap
\theta^{-1}(1)$ the mapping $\alpha_{a'} : \langle F \cup \{ b,a \}\rangle
\rightarrow \langle F \cup \{ b,a' \} \rangle$ which sends $f \mapsto
f$ (for $f \in F$), $b \mapsto b$, and $a \mapsto a'$, is an
isomorphism between finite connected subdigraphs of $D$. Hence by
$C$-homogeneity for each $a' \in A \cap \theta^{-1}(1)$ the
isomorphism $\alpha_{a'}$ extends to an automorphism of $D$. Now $a$
and $b$ belong to the same connected component of $\mathcal{B}_{1}$,
since they are both adjacent to $c$. Therefore for each $a' \in A \cap
\theta^{-1}(1)$ the vertices $a' = \alpha_{a'}(a)$ and $b =
\alpha_{a'}(b)$ belong to the same connected components of
$\mathcal{B}_{1}$ as one another. This implies that $\{ b \} \cup (A
\cap \theta^{-1}(1))$ is a subset of a single connected component of
$\mathcal{B}_{1}$. But this is impossible since by
Lemma~\ref{completebipartite} each component of $\mathcal{B}_{1}$ is
isomorphic to $K_{l,l}$, while $|\{ b \} \cup (A \cap \theta^{-1}(1))|
= l+1$. This is a contradiction and completes the proof of the
theorem.         
\end{proof}

The argument in Lemma~\ref{completebipartite} to exclude the
possibility that the reachability digraph is isomorphic to the
complement of a perfect matching can be adapted to show that the
reachability digraph of a highly arc-transitive 2-ended digraph cannot
be the complement of a perfect matching.  

\section{Connected-homogeneous digraphs with triangles}

In this section $D$ will be a connected locally-finite $C$-homogeneous
digraph with more than one end and we assume that
$D$ embeds a triangle. Since $D$
embeds a triangle it follows from Theorem~\ref{2_ended_classification} that $D$ is
not $2$-ended, and hence must have infinitely many ends. In this case
we are able to give
an explicit classification of the digraphs that arise. We now describe
this family of graphs, and then give a proof that any digraph
satisfying the above hypotheses belongs to this family. Our approach
is similar to that used
in \cite{Moller4} and \cite{Gray2}.

Recall from Section~2 that $D_3$ denotes the directed triangle.
For $r \in \mathbb{N}$ we use $T(r)$ to
denote the directed Cayley graph of the free product
\[
\langle a_1 \rangle * \langle a_2 \rangle * \cdots * \langle a_{r}
\rangle, \quad a_i^3=1, \quad i=1,2, \ldots, r
\]
of $r$ copies of the cyclic group $\mathbb{Z}_3$, with respect to the
generating set $A = \{ a_1, a_2, \ldots, a_{r} \}$. The digraph
$T(3)$ is shown in Figure~\ref{hello}. The corresponding undirected
graphs to these occur in the classification of
locally-finite distance-transitive graphs by Macpherson in
\cite{Macpherson1}. The main result of this section
is the following.

\begin{figure}
\[
\xymatrix{    
        & \node \arcc{rr} &        & \node  &        & \node \arcc{rr} &
     & \node  &        \\
\node \arcb{dd}  &        &        &        &        &        &        &
     & \node \arcc{dd}  \\
        &        & \node \arcc{rrrr} \arcb{dddrr} \arcc{dll} \arcb{ull}
\arcc{uul} \arcb{uur} &        &        &        & \node
\arcc{uul} \arcb{uur} \arcc{urr} \arcb{drr} &        &        \\
\node   &        &        &        &        &        &        &        &
\node  \\
        &        &        &        &        &        &        &        &
     \\
        &        & \node  &        & \node \arcb{uuurr} \arcb{ll}
\arcc{rr} \arcb{ddr} \arcc{ddl} &        & \node  &        &
\\
        &        &        &        &        &        &        &        &
     \\
        &        &        & \node \arcc{uul} &        & \node \arcb{uur} &
       &        &
}
\]
 \caption{A partial view of the digraph $T(3)$.} \label{hello}
\end{figure}
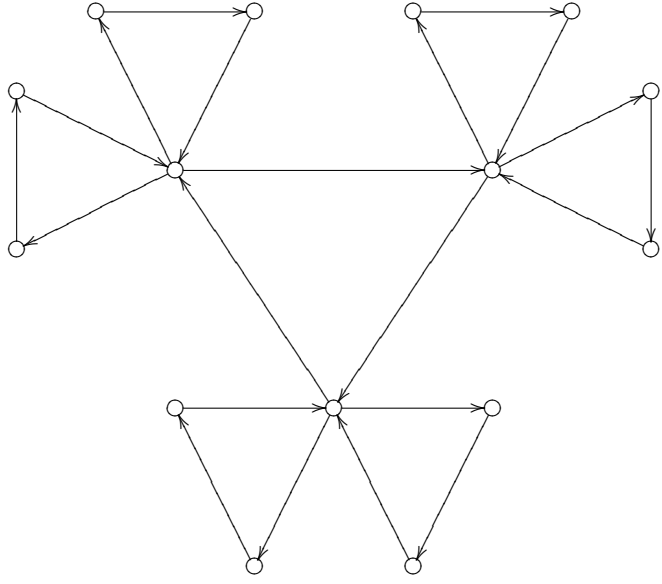

\begin{thm}\label{triangleclassification}
Let $D$ be a connected locally-finite digraph with more than one end, and
suppose that $D$ embeds a triangle. Then $D$ is $C$-homogeneous if and
only if it is isomorphic to $T(r)$ for some $r \in \mathbb{N}$.
\end{thm}

It is an easy exercise to check that the digraphs $T(r)$ are all
$C$-homogeneous. The rest of this section will be devoted to proving
the other direction of the theorem. 
For the remainder of this section $D$ will always denote a 
connected locally-finite digraph with more than one end such that $D$ embeds a triangle.

The following lemma is a straightforward consequence of the definitions.

\begin{lem}
For $v \in D$ the subdigraphs induced by $D^+(u)$ and $D^-(u)$ are
finite homogeneous digraphs.
\end{lem}
\begin{proof}
The subdigraph induced by $D^+(u)$ (respectively $D^-(u)$) is finite
since $D$ is locally finite.   Suppose $\varphi:U\rightarrow V$ is an
isomorphism between two induced subdigraphs of $D^+(u)$.  Let $U'$ be
the subdigraph of $D$ induced by the set $U\cup\{u\}$.
Define $V'$ similarly.  The isomorphism $\varphi$ extends to an
isomorphism  between $U'$ and $V'$ that maps $u$ to $u$.  The
subdigraphs $U'$ and $V'$ are connected and thus there is an
automorphism $\hat{\varphi}$ of $D$ that extends the isomorphisms $U'\rightarrow
V'$.  Since $\hat{\varphi}(u)=u$ we see that the restriction of
$\hat{\varphi}$ to $D^+(u)$ is an automorphism of $D^+(u)$ that extends
$\varphi$.  Whence $D^+(u)$ is homogeneous as claimed.  The proof that
$D^-(u)$ is homogeneous is similar.
\end{proof}

We use the same notation as in previous sections, with $e_0$ being a
fixed $D$-cut of $D$, $E = Ge_0
\cup Ge_0^*$ the associated tree set where $G = \Aut(D)$, $T(E)$ the structure tree, and
$\phi : VD \rightarrow VT$ the structure map. Fix a vertex $u \in
D$. Recall that we write $D(u)$ to denote all vertices that are joined
to $u$ by an arc.

\begin{lem}\label{morethanone}
$\phi(D(u))$ intersects more than one connected component of the graph
  induced by $T \setminus \{ \phi(u) \}$.
\end{lem}
\begin{proof}
If $\phi(D(u))$ were contained in a single connected component
of $T \setminus \{ \phi(u) \}$ then there would exist a cut $f \in E$
with $u \in f$ and $D(u) \subseteq f^*$. But this contradicts the fact
that $f$ is infinite and connected (see Theorem~\ref{DunwoodysThm} part (i)). 
\end{proof}

Since the digraph $D$ is $C$-homogeneous it follows that $G = \Aut{D}$
acts edge-transitively on the underlying undirected graph $\Gamma =
\Gamma(D)$. It follows that there is a positive integer $N$
such that whenever $u$ and $v$ are related vertices in $D$ then
$d_T(\phi(u),\phi(v)) = N$. This has the following consequence.

\begin{lem}\label{components}
If  $x$ and $y$ are vertices in $D(u)$ and $\phi(x)$ and $\phi(y)$
belong to different components of $T \setminus \{ \phi(u) \}$ then $x$ and
$y$ are not adjacent in $D$.
\end{lem}
\begin{proof}
In $T$ the unique shortest path between $\phi(x)$ and $\phi(y)$ must
pass through $\phi(u)$ because $\phi(x)$ and $\phi(y)$ belong to
different connected components of $T \setminus \{ \phi(u) \}$.  Hence
$d_T(\phi(x), \phi(y)) = 2N$, so $x$ and $y$ must be unrelated in
$D$.
\end{proof}

\begin{lem}\label{distributed}
Each of the sets $\phi(D^+(u))$ and $\phi(D^-(u))$ intersects more
  than one connected   component of $T\setminus \{ \phi(u) \}$.
\end{lem}
\begin{proof}
Suppose that $\phi(D^+(u))$ is contained in a single component
$\mathcal{C}$ of $T
\setminus \{ \phi(u) \}$.  By
Lemma~\ref{morethanone} there is a vertex $v \in D^-(u)$ with $\phi(v)
\not\in \mathcal{C}$. We know that $D$ embeds triangles so there is a
2-arc $(a,b,c)$ that is a part of some triangle.  By arc-transitivity
we can map the arc $a\rightarrow b$ to the arc $v\rightarrow u$.  Then
$c$ is mapped to a vertex in $D^+(u)$ and it follows that $v$ is
adjacent to a vertex in $D^+(u)$. But then, by Lemma~\ref{components},
$\phi(v) \in \mathcal{C}$ which is a contradiction. Similarly the
assumption that $\phi(D^-(u))$ is contained in a single component of
$T \setminus \{ \phi(u) \}$ leads to a contradiction.
\end{proof}

\begin{lem}\label{localstructure}
Suppose $x$ and $y$ are distinct vertices in $D(u)$.  Then
$x$ and $y$ are adjacent in $\Gamma(D)$ if  and only if
$\phi(x)$ and $\phi(y)$ belong to the same connected component of $T
\setminus \{ \phi(u) \}$.
\end{lem}
\begin{proof}
Let $x,y \in D(u)$ with $x$ and $y$ not adjacent in $D$. We claim that
$\phi(x)$ and $\phi(y)$ belong to different components of $T \setminus
\{ \phi(u) \}$.
Indeed, if they belonged to the same component $\mathcal C$
(say) then by Lemma~\ref{distributed}
there exists $b \in D(u)$  with $\phi(b)$ in a connected component
of $T \setminus \{ \phi(u) \}$ that is different from $\mathcal C$.
We can choose $b$ to be in $D^+(u)$ if $y$ is in $D^+(u)$ and in
$D^-(u)$ if $y$ is in $D^-(u)$.  Note that
$b$ is unrelated to $x$ (by Lemma~\ref{components})
and hence $\lb x,u,y \rb \cong \lb x,u,b
\rb$.  But then there is no automorphism extending the isomorphism between 
$\lb x,u,y \rb$ and $\lb x, u, b \rb$ that sends $y$ to $b$,
because $d_T(\phi(x),\phi(y))<2N$ while $d_T(\phi(x),\phi(b))=2N$.
This contradicts $C$-homogeneity and therefore the assumption that
$x$ and $y$ are not adjacent in $D$ cannot hold.

The other direction follows from Lemma~\ref{components}.
\end{proof}

\begin{lem}\label{nhoodisatrounament}
If the subdigraph induced by $D^+(u)$ (respectively $D^-(u)$)
is not a null graph then
it is isomorphic to the disjoint union of a
finite number of copies of the directed triangle $D_3$.
\end{lem}
\begin{proof}
By Lemma~\ref{localstructure}, for vertices in $D^+(u)$ the property
of being related by an arc is an equivalence relation. Thus, along
with arc-transitivity, this shows that $D^+(u)$ is a disjoint union of
isomorphic finite homogeneous tournaments. This proves the lemma since
the only finite homogeneous tournaments are the trivial one-element
graph, and the directed triangle $D_3$ (see \cite{Lachlan2}).
\end{proof}

\begin{lem}\label{no3chain}
Both $D^+(u)$ and $D^-(u)$ are independent sets.
\end{lem}
\begin{proof}
If $D^-(u)$ contains an arc $x \rightarrow y$ then $D^+(x)$ contains the arc $y \rightarrow u$,
and  by vertex-transitivity $D^+(u)$ contains an arc. Similarly, if $D^+(u)$ has
an arc then so does $D^-(u)$. Hence $D^+(u)$ is a null graph if and only if $D^-(u)$ is a null graph. 
Therefore, to prove the lemma it suffices to show that at least one of
$D^+(u)$ or $D^-(u)$ is an independent set.
Suppose, seeking a contradiction,
that neither $D^+(u)$ nor $D^-(u)$ is an independent set.  By
Lemma~\ref{nhoodisatrounament} each connected component of $D^+(u)$ and
$D^-(u)$ is isomorphic to a directed triangle.  Let $A$ be
a connected component of $D^-(u)$.  The argument in
Lemma~\ref{distributed} shows that some vertex in $A$ is adjacent to a
vertex $b$ in $D^+(u)$. Let $B$ denote the connected component of
$D^+(u)$ that $b$ belongs to.
  By Lemma~\ref{localstructure} both $A$ and $B$ are
mapped by $\phi$ to the same component of $T\setminus\{\phi(u)\}$.
Using Lemma~\ref{localstructure}
again we conclude that the subdigraph induced by
$A\cup B\cup\{u\}$ is a tournament $F$ with 7 vertices.
Note also that by Lemma~\ref{localstructure}
no vertex in $A$ is adjacent to a vertex in $D^+(u)$
outside $B$ and no vertex in $B$ is adjacent to a vertex in $D^-(u)$
outside $A$.

Suppose $A=\{a_1, a_2, a_3\}$ with $a_1\rightarrow a_2\rightarrow
a_3\rightarrow a_1$.  Note that both $D^-(a_1)\cap F$ and
$D^+(a_1)\cap F$ will be mapped
by $\phi$ to the same connected component of
$T\setminus\{\phi(a_1)\}$.  But only three vertices of each of
$D^-(a_1)$ and $D^+(a_1)$ can be mapped to this component and since $F$ is a tournament with $7$ vertices we
conclude that both $D^-(a_1)\cap F$ and $D^+(a_1)\cap F$ contain three
vertices.  The vertices $a_2$ and $u$ are in $D^+(a_1)\cap F$ and thus
there must be  a unique vertex $b'$ in $B$ such that $a_1\rightarrow b'$.
The same holds true for $a_2$ and $a_3$.
Since $F$ is a tournament and $b'$ is the unique vertex in $B$ with $a_1\rightarrow b'$ it follows that 
$b'\rightarrow a_2$ and $b'\rightarrow a_3$.
Then $\langle a_2, u, b' \rangle \cong
\langle a_3, u, b' \rangle\cong D_3$ and by $C$-homogeneity there is an
automorphism fixing $u$ and $b'$ and sending $a_2$ to $a_3$. This is
clearly a contradiction, since any automorphism fixing $u$ and $b'$,
must fix $B$ pointwise, and hence
also fixes $A$ pointwise.

Now we have reached a contradiction and have established the lemma.
\end{proof}

\begin{lem}\label{triangles}
The subdigraph induced by $D(u)\cup\{u\}$
is a union of $|D^+(u)|$ directed triangles
$D_3$ such that any two of them have just the vertex $u$ in common.
\end{lem}
\begin{proof}
Let $a$ be a vertex in $D^-(u)$.  By Lemma~\ref{no3chain} and the
assumption that $D$ embeds a triangle there is a
vertex $b$ in $D^+(u)$ such that $a$ and $b$ are related.  If
$a\rightarrow b$ is an arc in $D$ then the arc $u\rightarrow b$ is
contained in $D^+(a)$ contradicting the previous lemma.  Thus
$b\rightarrow a$ and $\langle a, u, b\rangle\cong D_3$.  We also see
that if $a$ were related to some other vertex $b'$ in $D^+(u)$ then,
by Lemma~\ref{localstructure},
$\phi$ would map $a, b$ and $b'$ all to the same component of
$T\setminus \{\phi(u)\}$ and thus, by Lemma~\ref{localstructure},
$b$ and $b'$  would be adjacent vertices in $D^+(u)$
contradicting Lemma~\ref{no3chain}.  Applying the same
argument to vertices in $D^+(u)$ we conclude that for each vertex
$b\in D^+(u)$ there is a unique vertex $a$ in $D^-(u)$ such that
$b\rightarrow a$.  Now we have proved that each vertex in $D(u)$ is in
a unique directed triangle containing $u$ and the lemma follows.
\end{proof}

\begin{proof}[Proof of Theorem~\ref{triangleclassification}]

Let $r=|D^+(u)|$.  We show that $D\cong T(r)$.  Fix a vertex $u'$ in
$T(r)$ and let $Y_j$ denote the subdigraph of $T(r)$ induced by
vertices in distance at most $j$ from $u'$.  Let $D_j$ denote the
subdigraph of $D$ induced by vertices in distance at most $j$ from
$u$.

Lemma~\ref{triangles} shows that $D_1$ is isomorphic to $Y_1$.
Let $\psi_1$ be an isomorphism between $D_1$ and $Y_1$.  Note that
$\psi_1(u)=u'$.
We use induction to construct a sequence of isomorphisms
$\psi_j:D_j\rightarrow Y_j$ such that if $i<j$ then $\psi_j$ agrees
with $\psi_i$ on $D_i$.  Then we define an isomorphism
$\psi:D\rightarrow T(r)$ such that if $v\in D_j$ then
$\psi(v)=\psi_j(v)$.

We already know that there is an integer $N$
such that if $x$ and $y$ are adjacent vertices in $D$ then
$d_T(\phi(x), \phi(y))=N$.  A preliminary step is to show that
if $x$ and $y$
are vertices in $D$ then $d_T(\phi(x), \phi(y))=d(x,y)N$, where
$d(x,y)$ denotes the distance between $x$ and $y$ in the underlying
undirected graph $\Gamma$.
Assume that $d_T(\phi(x), \phi(y))=Nd(x,y)$ whenever $d(x,y)\leq k$.
Suppose $d(x,y)=k+1$.  Find a path $x, x_1, \ldots, x_{k},
y$ of length $k+1$ in $D$.  By the induction hypothesis
$d_T(\phi(x), \phi(x_{k}))=kN$.  The vertices $x_{k-1}$ and $y$ are
not adjacent and are both  in $D(x_k)$.  Hence $\phi(x_{k-1})$ and
$\phi(y)$ belong to different components of
$T\setminus\{\phi(x_k)\}$ and the path in $T$ from $\phi(x_{k-1})$ to
$\phi(y)$ goes through $\phi(x_k)$ and has length $2N$.  Thus
$d_T(\phi(x), \phi(y))=N(k+1)=Nd(x,y)$.   In particular the map
$\phi$ is injective.

Suppose $y$ is a vertex in $D$ and
$d(u,y)=k+1$.  Let $x$ be a vertex in $D_k$ that is adjacent to $y$.
If $x'$ is another vertex in $D_k$ adjacent to $y$ then
$d_T(\phi(u), \phi(x))= d_T(\phi(u), \phi(x'))=kN$ and
$d_T(\phi(u), \phi(y))=(k+1)N$ and
$d_T(\phi(x), \phi(y))= d_T(\phi(x'), \phi(y))=N$.  But then both
$\phi(x)$ and $\phi(x')$ would be on the path from $\phi(u)$ to
$\phi(y)$ and we
also know that $\phi(x)\neq \phi(x')$.  This is clearly
impossible in a tree. Hence $y$ is adjacent to a unique vertex $x$ in
$D_k$.  We also see that if $y$ and $z$ are adjacent vertices in $D$
such that $d(u,y)=d(u,z)=k+1$ then there is a vertex $x$ in $D_k$ such
that both $y$ and $z$ are adjacent to $x$.

Suppose now that we have defined the isomorphism
$\psi_k:D_k\rightarrow Y_k$.    Let $x$ be a vertex in $D$ such that
$d(u,x)=k$.  The subgraph induced by $D(x)$ consists of $r$
triangles.  One of these triangles will be contained in $D_k$ and the
others will only have the vertex $x$ inside $D_k$.  The
same is true in relation to the vertex $\psi_k(x)$ and $Y_k$.
Now we can clearly extend $\psi_k$ to those vertices of $D(x)$
that are in distance $k+1$ from $u$.   Doing this for all the vertices
in $D_k$ that are in distance $k$ from $u$ will give us
the desired extension $\psi_{k+1}$ of $\psi_k$. 
\end{proof} 

\bibliographystyle{abbrv}

\begin{thebibliography}{10}

\bibitem{BrouwerBook}
A.~E. Brouwer, A.~M. Cohen, and A.~Neumaier.
\newblock {\em Distance-regular graphs}, volume~18 of {\em Ergebnisse der
  Mathematik und ihrer Grenzgebiete (3) [Results in Mathematics and Related
  Areas (3)]}.
\newblock Springer-Verlag, Berlin, 1989.

\bibitem{Cameron2}
P.~J. Cameron, C.~E. Praeger, and N.~C. Wormald.
\newblock Infinite highly arc transitive digraphs and universal covering
  digraphs.
\newblock {\em Combinatorica}, 13(4):377--396, 1993.

\bibitem{Cherlin1}
G.~L. Cherlin.
\newblock The classification of countable homogeneous directed graphs and
  countable homogeneous {$n$}-tournaments.
\newblock {\em Mem. Amer. Math. Soc.}, 131(621):xiv+161, 1998.

\bibitem{Dicks1}
W.~Dicks and M.~J. Dunwoody.
\newblock {\em Groups acting on graphs}, volume~17 of {\em Cambridge Studies in
  Advanced Mathematics}.
\newblock Cambridge University Press, Cambridge, 1989.

\bibitem{Moller3}
R.~Diestel, H.~A. Jung, and R.~G. M{\"o}ller.
\newblock On vertex transitive graphs of infinite degree.
\newblock {\em Arch. Math. (Basel)}, 60(6):591--600, 1993.

\bibitem{Dunwoody4}
M.~J. Dunwoody.
\newblock Accessibility and groups of cohomological dimension one.
\newblock {\em Proc. London Math. Soc. (3)}, 38(2):193--215, 1979.

\bibitem{Dunwoody2}
M.~J. Dunwoody.
\newblock Cutting up graphs.
\newblock {\em Combinatorica}, 2(1):15--23, 1982.

\bibitem{DunwoodyKron}
M.~J. Dunwoody, B.~Kr\"{o}n.
\newblock Vertex cuts.
\newblock arXiv:0905.0064

\bibitem{Enomoto1}
H.~Enomoto.
\newblock Combinatorially homogeneous graphs.
\newblock {\em J. Combin. Theory Ser. B}, 30(2):215--223, 1981.

\bibitem{evans}
D.~M. Evans.
\newblock Examples of {$\aleph\sb 0$}-categorical structures.
\newblock In {\em Automorphisms of first-order structures}, Oxford Sci. Publ.,
  pages 33--72. Oxford Univ. Press, New York, 1994.

\bibitem{Evans1}
D.~M. Evans.
\newblock An infinite highly arc-transitive digraph.
\newblock {\em European J. Combin.}, 18(3):281--286, 1997.

\bibitem{Gardiner2}
A.~Gardiner.
\newblock Homogeneity conditions in graphs.
\newblock {\em J. Combin. Theory Ser. B}, 24(3):301--310, 1978.

\bibitem{Goldstern1}
M.~Goldstern, R.~Grossberg, and M.~Kojman.
\newblock Infinite homogeneous bipartite graphs with unequal sides.
\newblock {\em Discrete Math.}, 149(1-3):69--82, 1996.

\bibitem{Gray2}
R.~Gray.
\newblock $k$-${CS}$-transitive infinite graphs.
\newblock {\em J. Combin. Theory Ser. B}, 99:378--398, 2009.

\bibitem{Gray1}
R.~Gray and D.~Macpherson.
\newblock Countable connected-homogeneous graphs.
\newblock {\em J. Combin. Theory Ser. B}, 100:97--118, 2010. 

\bibitem{Hamann}
M.~Hamann, F.~Hundertmark.
\newblock A classification of connected-homogeneous digraphs.
\newblock arXiv:1004.5273

\bibitem{Lachlan3}
A.~H. Lachlan.
\newblock Finite homogeneous simple digraphs.
\newblock In {\em Proceedings of the Herbrand symposium (Marseilles, 1981)},
  volume 107 of {\em Stud. Logic Found. Math.}, pages 189--208, Amsterdam,
  1982. North-Holland.

\bibitem{Lachlan2}
A.~H. Lachlan.
\newblock Countable homogeneous tournaments.
\newblock {\em Trans. Amer. Math. Soc.}, 284(2):431--461, 1984.

\bibitem{Macpherson1}
H.~D. Macpherson.
\newblock Infinite distance transitive graphs of finite valency.
\newblock {\em Combinatorica}, 2(1):63--69, 1982.

\bibitem{Malnic2}
A.~Malni{\v{c}}, D.~Maru{\v{s}}i{\v{c}}, R.~G. M{\"o}ller, N.~Seifter,
  V.~Trofimov, and B.~Zgrabli{\v{c}}.
\newblock Highly arc transitive digraphs: reachability, topological groups.
\newblock {\em European J. Combin.}, 26(1):19--28, 2005.

\bibitem{Malnic1}
A.~Malni{\v{c}}, D.~Maru{\v{s}}i{\v{c}}, N.~Seifter, and B.~Zgrabli{\'c}.
\newblock Highly arc-transitive digraphs with no homomorphism onto {$\Bbb Z$}.
\newblock {\em Combinatorica}, 22(3):435--443, 2002.

\bibitem{Mansilla1}
S.~Mansilla.
\newblock An infinite family of sharply two-arc transitive digraphs.
\newblock In {\em European Conference on Combinatorics, Graph Theory and
  Applications, European Conference on Combinatorics, Graph Theory and
  Applications Sevilla, Spain 11-15 September 2007}, volume~29 of {\em Stud.
  Logic Found. Math.}, pages 243--247. Elsevier, 2007.

\bibitem{Moller4}
R.~G. M{\"o}ller.
\newblock Distance-transitivity in infinite graphs.
\newblock {\em J. Combin. Theory Ser. B}, 60(1):36--39, 1994.

\bibitem{Moller5}
R.~G. M{\"o}ller.
\newblock Groups acting on locally finite graphs---a survey of the infinitely
  ended case.
\newblock In {\em Groups '93 Galway/St.\ Andrews, Vol.\ 2}, volume 212 of {\em
  London Math. Soc. Lecture Note Ser.}, pages 426--456. Cambridge Univ. Press,
  Cambridge, 1995.

\bibitem{Moller6}
R.~G. M{\"o}ller.
\newblock Descendants in highly arc transitive digraphs.
\newblock {\em Discrete Math.}, 247(1-3):147--157, 2002.

\bibitem{Praeger1}
C.~E. Praeger.
\newblock Highly arc transitive digraphs.
\newblock {\em European J. Combin.}, 10(3):281--292, 1989.

\bibitem{Seifter1}
N.~Seifter.
\newblock Transitive digraphs with more than one end.
\newblock {\em Discrete Math.}, 308(9):1531--1537, 2008.

\bibitem{Thomassen1}
C.~Thomassen and W.~Woess.
\newblock Vertex-transitive graphs and accessibility.
\newblock {\em J. Combin. Theory Ser. B}, 58(2):248--268, 1993.

\end{thebibliography}

\end{document}